\documentclass{amsart}

\usepackage{changes}

\allowdisplaybreaks[4]

\usepackage[margin=2cm]{geometry}

\usepackage[utf8]{inputenc}
\usepackage{bbm}
\usepackage{amsmath,amsfonts,amssymb,amsthm,amscd,comment,euscript,geometry,graphicx}

\usepackage{tikz}

\usepackage{tikz-cd}

\usepackage{caption}

\graphicspath{{./pictures/}}

\DeclareMathOperator{\wt}{wt}
\DeclareMathOperator{\tr}{tr}

\usepackage{xcolor}

\usepackage{mathbbol, stmaryrd}

\newtheorem{thm}{Theorem}[section]

\newtheorem{lem}[thm]{Lemma}
\newtheorem{prop}[thm]{Proposition}
\newtheorem{ex}[thm]{Example}
\newtheorem{conj}[thm]{Conjecture}
\newtheorem{rem}[thm]{Remark}
\newtheorem{dfn}[thm]{Definition}

\newcommand{\Q}{\mathbb{Q}}

\newcommand{\C}{\mathbb{C}}
\newcommand{\Z}{\mathbb{Z}}

\newcommand{\al}{\alpha}
\newcommand{\be}{\beta}
\newcommand{\la}{\langle}
\newcommand{\ra}{\rangle}
\newcommand{\La}{\Lambda}
\newcommand{\s}{\mathfrak{sl}}

\newcommand{\cmmnt}[1]{}
\newcommand{\fwt}[1]{\overline{\Lambda}_{#1}}

\newcommand{\bm}[1]{\mathbf{#1}}

\newcommand{\ju}{\;\;\;}
\newcommand{\bd}[1]{\mathbf{#1}}
\newcommand{\dlm}{\text{dlm}}
\newcommand{\iso}{\stackrel{\sim}{\smash{\longrightarrow}\rule{0pt}{0.4ex}}}

\newcommand{\f}[2]{\frac{#1}{#2}}
\newcommand{\res}[2]{\text{Res}_{#1}{#2}}
\newcommand{\p}[1]{\left(#1\right)}

\newcommand{\te}[1]{\text{#1}}
\newcommand{\h}{\mathfrak h}

\def \<{\langle}
\def \>{\rangle}

\def \a{\alpha }

\usepackage{hyperref} 
\hypersetup{
           colorlinks=true,
           linkcolor=blue,
           urlcolor=red,
           pdftitle={Fusion rules},
           }

\title{Spectral flow, twisted modules and MLDE of quasi-lisse vertex algebras}

\author{Bohan Li, Hao Li, Wenbin Yan}
\address{Yau Mathematics Sciences Center, Tsinghua University,  Beijing, 100084,China}
\email{libh19@mails.tsinghua.edu.cn}
\email{haoli2021@mail.tsinghua.edu.cn}
\email{wbyan@mail.tsinghua.edu.cn}

\begin{document}
\begin{sloppypar}
\maketitle

\begin{abstract}
We calculate the fusion rules among $\Z_2$-twisted modules $L_{\s_2}(\ell,0)$ at admissible levels. We derive a series MLDEs for normalized characters of ordinary twisted modules of quasi-lisse vertex algebras. Examples include affine VOAs of type $A_1^{(1)}$ at boundary admissible level,  admissible level $k=-1/2$, $A^{(1)}_{2}$ at boundary admissible level $k=-3/2$, and $\mathrm{BP}^{k}$-algebra with special value $k=-9/4$. We also derive characters of some non-vacuum modules  for affine VOA of type $D_4$ at non-admissible level $-2$ from spectral flow automorphism.
\end{abstract}

\tableofcontents

\newpage

\section{Introduction}
Four-dimensional $\mathcal{N}=2$ superconformal field theories (SCFTs) in physics have rich mathematical structures. In \cite{Beem:2013sza}, the authors propose a correspondence between the Schur sectors of $4d$ $\mathcal{N}=2$ SCFTs and $2d$ vertex operator algebras (VOAs). This correspondence has fueled a lot of work in the past  years, including some conjectures about the chiral algebra in the context of theories of Class $\mathcal{S}$ \cite{Gaiotto:2009we,Gaiotto:2009hg,Beem:2014rza,Lemos:2014lua}. For the genus zero case, the conjecture has been proved in terms of a functorial construction \cite{Arakawa:2018egx}. Class $\mathcal{S}$ theories have the Coulomb branch operators with integral scaling dimension, yet there is another class of $\mathcal{N}=2$ SCFTs called Argyres-Douglas (AD) theories \cite{Argyres:1995jj,Xie:2012hs} which usually have fractional scaling dimensions for the Coulomb branch operators. These AD theories can be constructed by compactifying $6d$ $(2,0)$  theory on a Riemann surface with irregular singularities. The corresponding VOAs of a class of AD theories are identified with certain affine Kac-Moody algebras $L_{k}(\mathfrak{g})$ at admissible level $k$, or affine $\mathcal{W}_{k}(\mathfrak{g},f)$-algebras \cite{Cordova:2015nma,Buican:2015ina,Buican:2015tda,Song:2017oew,Xie:2016evu,Xie:2019yds}. Dualities of $4d$ theories imply nontrivial isomorphism and collapsing levels of VOAs \cite{Xie:2019yds, Xie:2019vzr,Li:2022njl}, some of which were proved rigorously \cite{Arakawa:2021ogm,adamovic2021certain}.

One consequence of this SCFT/VOA correspondence is that the Schur index of the 4d SCFT is equal to the normalized vacuum character of the corresponding VOA, hence the character formula provides a valuable tool to study the spectrum of 4d SCFTs. The character here means that the trace $\chi_{\lambda}(\tau,z)=\tr_{L(\ell,\lambda)}e^{2\pi i\tau(L(0)-\f12z h(0)-\f{1}{24}c_{\ell})}$ over the $L(\ell,\lambda)$. In \cite{kac1988modular}, the authors derived  character formulas for admissible representations of an affine Kac-Moody Lie algebra $\hat{\mathfrak{g}}$ at a rational level $\ell$, i.e., $L(\ell,\lambda)$,
and also investigated the modular  property of these characters. In particular, the transformed character $\chi_{\lambda}(-\f{1}{\tau},\f{z}{\tau})$ can be written as a linear combination of characters of admissible representations with a shifted conformal vector, while $\chi_{\lambda}(-\frac{1}{\tau},z)$ is a linear combination of the characters of some $\Z_2$-twisted modules. The character formulas were used to derive the Schur index of a large class of AD theory \cite{Xie:2019zlb}, while the modular properties also have applications in physics \cite{Razamat:2012uv}.

Another conjecture of the SCFT/VOA correspondence is the identification between the Higgs branch of vacua of an $\mathcal{N}=2$ SCFT and the associated variety of the corresponding VOA \cite{Beem:2017ooy, Song:2017oew}, which were also used to propose lisse VOAs from $4d$ SCFTs \cite{Xie:2019vzr}.  The VOA corresponding to a 4d SCFT is often of the quasi-lisse type whose associated variety has finitely many symplectic leaves. The normalized character of an ordinary representation of a quasi-lisse VOA was shown to satisfy a modular linear differential equation (MLDE), and solving MLDE gives explicit expression for the characters of the affine Lie algebra of the Deligne-Cvitanovic (DC) series \cite{Arakawa:2016hkg}.  The MLDE for VOAs corresponding to several families of AD theories and $\mathcal{N}=4$ super Yang-Mills with $\mathfrak{su}(n)$ gauge group were also discussed in \cite{Beem:2017ooy}. Recently in \cite{Zheng:2022zkm} the authors constructed  flavored MLDEs for the Schur index of $\mathfrak{a}_1$ Class $\mathcal{S}$ theories based on a compact formula for the index they found earlier \cite{Pan:2021mrw}. Both works used the Higgs branch structure to probe the singular vector of the corresponding VOA and then derive the MLDE.

The SCFT/VOA correspondence also goes beyond the Schur index and the vacuum module.  One generalization is to consider the lens space index \cite{Razamat:2013jxa} instead of the normal index. In \cite{Fluder:2017oxm}, the lens space index was identified with the characters of the twisted modules. Given an automorphism $g$ of a VOA $V$ of finite order, the basic properties of  $g$-twisted modules for $V$ are systematically studied by Haisheng Li in \cite{li1996local} using the twisted local systems. In particular, he showed that $\Z_2$-twisted modules of the affine VOA mentioned above can be obtained from its untwisted modules via Li's $\Delta$-operators. Later on, in the important work \cite{dong2000modular}, authors showed that the trace functions of the $g$-twisted modules for the $C_2$-cofinite rational VOA satisfy certain twisted modular linear differential equations (MLDE) and possess some modular invariance properties. In \cite{Hao2022}, the author generalized some of the results obtained by Dong, Li, and Mason to the quasi-lisse vertex operator (super)algebras case, as while as proved that characters of twisted modules of quasi-lisse vertex algebras satisfy certain twisted MLDEs.  Another generalization is to consider the index in the presence of defects, which were identified with the twisted modules using spectral flow \cite{Cordova:2015nma,Cordova:2016uwk,Cordova:2017mhb}.   The spectral flow used in these works has a long history in the conformal field theory (see e.g. \cite{Schwimmer:1986mf,Lesage:2002ch,Ridout:2008nh,Ridout:2010qx,Creutzig:2012sd,Creutzig:2013yca,Kawasetsu:2021qls}).

As reviewed above, there is an increasing interests in understanding twisted modules and spectral flowed modules of VOAs from the perspective of SCFT/VOA correspondence, as it provides the knowledge of lens space indices and  defects of the corresponding SCFTs. Especially if one can show that characters of these modules are solutions of certain MLDEs, closed form expression might also within the grasp as in the ordinary module case. Since conventional methods in physics usually give only power series, such closed form expressions are valuable and may review more interesting properties.
 Mathematically, people introduce the twisted modules of $V$ to study the module category of $G$ invariants $V^{G}$, where $G$ is finite automorphism group of $V$. In the present work we study the $\Z_2$-twisted modules denoted by $\sigma^{-\f12}(L(\ell,\lambda))$ for affine VOA $L_{\mathfrak{sl}_{2}}(\ell,0)$ associated with Lie algebra $\mathfrak{sl}_{2}$ at admissible level and the fusion rules among them  by using the twisted Zhu's bimodule recently introduced by \cite{zhu2022bimodues} and a conjectural twisted version of Frenkel-Zhu's bimodule theorem. One of our main results is the fusion rules among admissible modules and twisted modules. Firstly, we showed:
\begin{thm}
Let $\ell=-2+\f{p}{q}$ be the admissible level where $p$ and $q$ coprime positive integers with $p\geq 2$. Then $\Z_2$-twisted Zhu's algebra for $L_{\ell}(\s_2)$ is
$\C[x]/\la\prod_{r=0}^{p-2}\prod_{s=0}^{q-1}(x+\f12\ell-r+st)\ra$,
where $0\leq r\leq p-2$, $0\leq s\leq q-1$. In particular, the Dynkin labels (eigenvalues of  $h_{(0)} $ on the highest weight vector) of all $\Z_2$-twisted modules in category $\mathcal{O}$ are $\{r-st-\f12\ell|0\leq r\leq p-2,0\leq s\leq q-1\}$.
\end{thm}
 \begin{thm}
	All  irreducible $\mathbb{Z}_{2}$-twisted modules of $L_{k}(\mathfrak{sl}_2)$ at admissible level in category $\mathcal{O}$ can be obtained by using $\ell=-\frac{1}{2}$ spectral flow on the untwisted modules in category $\mathcal{O}$. In, particular, all of those irreducible twisted modules are ordinary modules at boundary admissible level.

\end{thm}

\noindent Then we obtain the fusion rules among these $\Z_2$-twisted modules.
\begin{thm}
  For admissible weights $j_i=n_i-(k_i-1)t$ $(i=1,2)$ of the vertex affine algebra $L_{k}(\mathfrak{sl}_{2})$, there are following fusion rules between one admissible module and one twisted module:
  \begin{align} L(k,j_1)\times \sigma^{-\f12}(L(k,j_2))=\sum_{i=\max\{0,n_1+n_2-p\}}^{\min\{n_1-1,n_2-1\}}\sigma^{-\f12}(L(k,j_1+j_2-2i))\end{align}
 if $0\leq k_2-1\leq q-k_1$, and $L(k,j_1)\times \sigma^{-\f12}(L(k,j_2))=0$ otherwise.
\end{thm}

\begin{rem}
According to the results in \cite{xu1995intertwining}, one can generalize the symmetries of fusion rules in \cite{frenkel1993axiomatic} to the twisted case:
\begin{align*}
	&N_{jk}^{i}=N_{kj}^{i}\\
	&N_{jk}^{i}=N_{ji}^{k}.
\end{align*}
Therefore, using above results one can obtain the fusion rules for $\sigma(L )\times \sigma(L)$.
\end{rem}

 \noindent We also prove the following theorem.

\begin{thm}
\label{thm:mainThm2}
The category of the ordinary $\Z/2\Z$-twisted modules vertex operator algebra  $L_{\ell}(\mathfrak{sl}_{2})$ at the boundary admissible level $\ell=-2+\frac{2}{q}$ $(\mathrm{gcd}(q,2)=1)$ is semi-simple.
\end{thm}

 \noindent Following the idea in \cite{dong1997vertex} we compute the fusion rules among  admissible modules of $L_{k}(\s_2)$ and their contragredient modules.

 \begin{thm}
For admissible weight $j_i=n_i-1-(k_i-1)t$ $(i=1,2)$ of the vertex affine algebra $L_{k}(\mathfrak{sl}_{2})$, there are following fusion rules:
  \begin{align}
   & L(k,j_1)\times L(k,j_2)=\sum_{i=\max\{0,n_1+n_2-p\}}^{\min\{n_1-1,n_2-1\}}L(k,j_1+j_2-2i), \label{fusion3}\\
   & (L(k,j_1))^*\times L(k,j_2)
   =L(k,j_2)\times (L(k,j_1))^*=
  \left\{
    \begin{array}{ll}
      L(k,-j_1+j_2) , & \hbox{if $n_2-n_1\geq 0$;} \\
      (L(k,j_1-j_2))^*, & \hbox{if $n_2-n_1<0$.}
    \end{array}
  \right. \label{fusion2}\\
   & (L(k,j_1))^*\times (L(k,j_2))^*
   =\sum_{i=\max\{0,n_1+n_2-p\}}^{\min\{n_1-1,n_2-1\}}(L(k,j_1+j_2-2i))^*.\label{fusion4}
   \end{align}
 where $(L(k,j))^*$ denote the contragredient module of the irreducible highest weight module $L(k,j)$.
\end{thm}

\noindent We further use this result to obtain the fusion rules among $\Z_2$-twisted modules and their contragredient modules:
\begin{thm}
\begin{align}
     &  (L(\ell,j_1))^*\times \sigma^{\f12}((L(\ell,j_2))^*)
 =\sum_{i=\max\{0,n_1+n_2-p\}}^{\min\{n_1-1,n_2-1\}}\sigma^{\f12}((L(\ell,j_1+j_2-2i))^*).\\
      & (L(\ell,j_1))^*\times \sigma^{\f12}(L(\ell,j_2))
   =
  \left\{
    \begin{array}{ll}
      \sigma^{\f12}(L(\ell,-j_1+j_2)) , & \hbox{if $n_2-n_1\geq 0$;} \\
      \sigma^{\f12}((L(\ell,j_1-j_2))^*), & \hbox{if $n_2-n_1<0$.}
    \end{array}
  \right.\\
      &
   L(\ell,j_2)\times \sigma^{-\f12}((L(\ell,j_1))^*)=
  \left\{
    \begin{array}{ll}
      \sigma^{-\f12}(L(\ell,-j_1+j_2)) , & \hbox{if $n_2-n_1\geq 0$;} \\
      \sigma^{-\f12}((L(\ell,j_1-j_2))^*), & \hbox{if $n_2-n_1<0$.}
    \end{array}
  \right.
  \end{align}
\end{thm}

We show that the ordinary twisted modules satisfy the twisted MLDEs. 
For $A^{(1)}_{1}$ with boundary admissible level,  if we only consider $q$-series, the normalized characters of those modules form the complete solutions of a $(u+1)/2$-order $\Gamma^{0}(2)$ MLDE. For $A^{(1)}_{2}$ at boundary admissible level, we construct some ordinary $e^{2\pi i v_{(0)}}$-twisted modules by the spectral flow of untwisted modules along different directions in the weight lattice (see Section 2.5), whose normalized characters satisfy a second-order $\Gamma^{0}(2)$ MLDE.  For $\mathrm{BP}^{k}$-algebra with $k=-9/4$, we show that characters of twisted modules are solutions of a third-order MLDE under full $SL(2,\mathbb{Z})$ group.  Finally we study characters of $\mathbb{Z}_{2}$-twisted modules and spectral flowed modules of the affine vertex algebra $\mathcal{L}_{\mathfrak{d}_{4}}(-2,0)$. Since it is non-admissible, one  can not use Kac-Wakimoto formula to write down the characters of simple modules directly.  
 Fortunately, we can get some simple modules of $\mathcal{L}_{\mathfrak{d}_{4}}(-2,0)$ from spectral flowed modules. We also get one ordinary $\mathbb{Z}_{2}$-twisted module, whose character satisfies a second-order $\Gamma^{0}(2)$-MLDE. In general, for the simple affine vertex algebra $L_{k}(\mathfrak{g})$ with $\mathfrak{g}$ being a Lie algebra of DC series and $k=-h^{\vee}/6-1$, we make the following conjecture on the MLDE of characters of its $\mathbb{Z}_{2}$-twisted modules.

\begin{conj}\label{conjecture}
Let $V$ be a simple affine vertex algebras associated with the $\mathrm{DC}$ series, $$A_{1}\subset A_{2} \subset G_{2} \subset D_{4} \subset F_{4}\subset E_{6} \subset E_{7} \subset E_{8}$$ at level $-h^{\vee}/6-1$. The normalized $\mathbb{Z}_{2}$-twisted characters $\chi^{\mathrm{twi}}(q)$ are solutions of a second-order $\Gamma^{0}(2)$-$\mathrm{MLDE}$ with suitable coefficients $a_{1}$ and $a_{2}$ but without a $D^{(1)}_{q}$ term,
\begin{align}
\left(D^{(2)}_{q}+a_{1}\Theta_{0,2}+a_{2}\Theta_{1,1}\right)\chi^{\mathrm{twi}}(q)=0
\end{align}
\end{conj}
\noindent We show that this conjecture is true when $\mathfrak{g}$ is a classical Lie algebra.

\begin{rem}\label{remark}
For affine VOA of type $A_{1}$ and $A_{2}$ at boundary admissible level, one can get ordinary $\mathbb{Z}_{2}$-twisted module $\sigma^{-\frac{1}{2}\Lambda}(M_{\mathrm{vac}})$ by taking a spectral flow on the vacuum module $M_{\mathrm{vac}}$ along a special direction $-\frac{1}{2}\Lambda$ in the  weights lattice. We believe that this phenomenon is universal for all affine VOAs associated with the semisimple Lie algebra $\mathfrak{g}$ at admissible level $k$.  So far, for a general VOA $V$, we do not have a systematic approach to get the ordinary $\mathbb{Z}_{2}$-twisted modules from irreducible non-twisted $V$-modules in the category $\mathcal{O}$.
\end{rem}


In this work, the usage of the spectral flow is crucial, as we use it to obtain twisted modules and new modules. In particular, we hope to prove Conjecture \ref{conjecture} and address some of the questions raised in above Remark \ref{remark} in the future work. It is also interesting to investigate the tensor category structures for some subcategory of the relaxed highest weight category for the affine VOA at fractional  level.


We summarize the main contents in this paper.
In Section \ref{section2}, we recall some basic notions of admissible representations, their (twisted) characters, the modularity of their characters under the transformations $\tau\rightarrow -\f{1}{\tau}$ , and semi-simplicity of the $\Z/2\Z$-twisted modules of $L_k(\s_2)$ at boundary admissible level.
 In Section \ref{section3}, we study the twisted Zhu's bimodules of the $\Z_2$-twisted modules of $L_k(\s_2)$ at admissible level, then use it to classify twisted modules and compute their fusion rules.  In Section \ref{section4}, we further calculate  fusion rules among the highest weight  modules and their contragredient modules for $L_k(\s_2)$ at admissible level. In Section \ref{section5}, we study the $e^{2\pi i v_{(0)}}$-twisted modules of affine VOAs at admissible level, then derive MLDEs their characters satisfy. In Section \ref{section6}, we discuss on the generalization of the result in the previous section to  $\mathcal{L}_{\mathfrak{d}_{4}}(-2\Lambda_{0})$.

\section*{Acknowledgement}

The authors wish to express their gratitude to Peng Shan for numerous fruitful discussions. HL would like to thank Antun Milas for some motivating questions and suggestions back in 2021. The authors would also like to thank Tomoyuki Arakawa, Thomas Creutzig, Antun Milas for valuable suggestions to improve some statements in the paper. BL, HL and WY are supported by  national key research
and development program of China (NO. 2020YFA0713000) and Dushi program of Tsinghua. WY is also supported by NNSF of China with
Grant NO: 11847301 and 12047502.


%
%
\section{Character formulae at admissible level and their modularity}\label{section2}
  Let $\hat{\Delta}$, $\hat{\Delta}_{+}$, $\hat{\Delta}_{+}^{re}$  be all roots, positive roots, positive real roots of an affine Lie algebra $\hat{\mathfrak{g}}$. Let $\ell=\text{dim}(\mathfrak{g})$.
Given $\lambda\in \mathfrak{h}^*$, the set of  $\lambda$-integral roots is   $\hat{\Delta}^{\lambda}:=\{\al\in \hat{\Delta}^{re}|\lambda(\al^{\vee})\in \mathbb{Z}\}$.   $\lambda$ is  an {\em admissible weight} if the following two properties hold
\begin{itemize}
    \item $(\lambda+\rho)(\al^{\vee})\notin \mathbb{Z}_{\leq 0}$ for all $\al\in \hat{\Delta}_{+}^{re}$,
    \item $\mathbb{Q}\hat{\Delta}^{\lambda}=\mathbb{Q}\hat{\Delta}.$
\end{itemize}
Roughly speaking, the admissible weight is integrable with respect to $\hat{\mathfrak{g}}_{\hat{\Delta}^{\lambda}}$ after a shift by the Weyl vector $\rho$ of $\hat{\mathfrak{g}}$. If there exists a further isometry $\phi$ of $\mathfrak{h}^*$ such that $\phi(\hat{\Delta}^{\lambda})=\hat{\Delta}$,  $\lambda$ is called a {\em principal} admissible weight, i.e., a principal admissible weight is integrable with respect to $\hat{\mathfrak{g}}$ after a shift of a Weyl vector.

\begin{thm}\cite{kac1988modular}
Let $\hat{\mathfrak{g}}$ be a Kac-Moody Lie algebra with a symmetrizable generalized Cartan matrix, and $\lambda\in \mathfrak{h}^{*}$ be an admissible weight. Then the character $L(\lambda)$ is given by the following formula:
\begin{align}\label{cfaal1}
    \mathrm{ch}L(\lambda)=\frac{1}{R}\cdot \displaystyle\sum_{w\in W_{\lambda}}\varepsilon(w) e^{w(\lambda+\rho)},
 \end{align}
where $R:=e^{\rho}\prod_{n=1}^{\infty}(1-q^{n})^{\ell}\prod_{\al\in\hat{\Delta}_{+}}(1-e^{\al}q^{n})(1-e^{-\al}q^{n-1})$ is the Kac-Weyl denominator, and $W^{\lambda}:=\{r_{\al}|\al\in \hat{\Delta}^{\lambda}\}$.
\end{thm}

From now on we focus on the case of $A_{1}^{(1)}$ until stated otherwise. It is known that all admissible weights of $A_{1}^{(1)}$ are principal admissible. The level $\lambda(c)=k=\frac{t}{u}$ of an admissible weight satisfies the following condition:
\begin{align}\label{cfaal11}
k+h^{\vee}\geq \frac{h^{\vee}}{u}\quad \text{and}\quad \text{gcd}(u,h^{\vee})=\gcd(u,r^{\vee})=1,
\end{align}
with the dual Coxeter number $h^\vee=2$ and the lacety number $r^\vee=1$.
All admissible weights at level $m=\frac{t}{u}$ are given by \cite{kac1988modular}:
\begin{align*}
    P^{m}&=\{\lambda_{m,k,n}:=(m-n+k(m+2))\Lambda_0+(n-k(m+2))\Lambda_1|n,k\in \mathbb{N}, n\leq 2u+t-2,k\leq u-1.\}\\
         &=\{t_{-\frac{k\al}{2}}(\widetilde{\Lambda^0}-(u-1)(m+2)\Lambda_0\},
\end{align*}
where $\widetilde{\Lambda^0}=(u(m+2)-2-n)\Lambda_0+n\Lambda_1.$

One can write down the normalized character $\chi_{\lambda}(\tau,z,t)$ for $L(\lambda_{m,k,n})$  \cite{kac1988modular}:
\begin{align}
    \begin{split}
       \chi_{\lambda}(\tau, z, t) &=\frac{A_{\lambda+\rho}(h)}{A_{\rho}(h)}\\
                                  &=\frac{(\Theta_{a^{+},b}-\Theta_{a^-,b})(\tau,\frac{z}{u},\frac{t}{u^2})}{(\Theta_{1,2}-\Theta_{-1,2})(\tau,z,t))}\\
                                  &=\frac{(\Theta_{a^{+},b}-\Theta_{a^-,b})(\tau,\frac{z}{u},\frac{t}{u^2})}{-ie^{-4\pi i t}\vartheta_{11}(\tau,z)}
    \end{split}
\end{align}

\noindent where theta functions $\Theta_{m,n}$ are defined in Appendix and $A_\lambda$ is defined as \begin{align}\label{impfor}
A_{\lambda}(h):=\sum_{w\in W^{\lambda}}\varepsilon(w)\Theta_{w(\lambda)}(h).
\end{align}
Here $W^{\lambda}:=\{r_{\al}|\al\in {\hat{\Delta}^{\lambda}}\}$, and
\begin{align*}
    a^{+}:=u((n+1)-k(m+2)),\quad a^-:=u(-(n-1)-k(m+2)),\quad b:=u^2(m+2).
\end{align*}

The modular $S$ transformation property of $\chi_{\lambda}$ is
\begin{align}\label{modulartrans}
    \begin{split}
      \chi_{\lambda_{m,k,n}}(-\frac{1}{\tau},\f{z}{\tau},t-\f{|z|^{2}}{2\tau})=\displaystyle\sum_{\substack{0\leq k' \leq u-1\\ 0\leq n'\leq u(m+2)-2}}a^{(m)}_{(k,n),(k',n')}\chi_{\lambda_{m,k',n'}}(\tau,z,t),
    \end{split}
\end{align}
where
\begin{align}
    a_{(k,n),(k',n')}^{(m)}:=\sqrt{\frac{2}{u^2(m+2)}}e^{i\pi (k'(n+1)+k(n'+1))}e^{-i\pi k k'(m+2)}\times \sin \frac{(n+1)(n'+1)\pi}{m+2}.
\end{align}

\subsection{Twisted modular transformation}

Let $N:=\overline{\mathfrak{h}}_{\mathbb{R}}\times\overline{\mathfrak{h}}_{\mathbb{R}} \times i \mathbb{R}$ be the Heisenberg group with multiplication
\begin{align}
    (\al,\be,u)\cdot (\al',\be',u'):=(\al+\al',\be+\be',u+u'+\pi i (\< \al|\be'\>-\<\a',\be\>)).
\end{align}
One defines the action of modular group $SL_2(\Z)$ and Heisenberg group $N$ on the space $\h^*$ as follows:
\begin{align}
    \begin{split}
    \begin{pmatrix}a & b \\c & d \end{pmatrix}\cdot(\tau,z,t);=(\frac{a\tau+b}{c\tau+d},\frac{z}{c\tau+d},t-\frac{c|z|^2}{2(c\tau+d)}),
    \end{split}
\end{align}
and
\begin{align}\label{heisen}
    (\a,\be,u)\cdot h:=t_{\be}h+2\pi i \a+(u-\pi i \<\a,\be\>)\delta,
\end{align}
where $t_{\be}$ is the translation operator.

\noindent One can check that the action of Heisenberg group $N$ on $\h^*$ is an group action, i.e.,
\begin{align}
    ((\a,\be,u)(\al',\be',u'))\cdot h=(\al,\be,u)\cdot((\al',\be',u')\cdot h).
\end{align}

\begin{lem}\cite{wakimoto2001infinite}
For $(\a,\be,u)\in N$ and $h=(\tau,z,t)=2\pi i (-\tau \Lambda_0+z+t\delta)\in \h^*,$
the following holds:
\begin{align}
    (\a,\be,u)\cdot(\tau,z,t)=(\tau,z+\a-\tau\be,t+\frac{u}{2\pi i}-\frac{\<\a,\be\>}{2}+\frac{\tau}{2}|\be|^2-\<\be,z\>).
\end{align}
\end{lem}
\begin{proof}
By direct calculation, one has
\begin{align*}
  \begin{split}
     & t_{\be}(\La_0)=\La_0+\be-\frac{|\be|^2}{2}\delta,\\
     & t_{\be}(\delta)=\delta,\\
     & t_{\be}(z)=z-\<z,\be\>\delta.
  \end{split}
\end{align*}

\noindent Thus,
\begin{align*}
    t_\be(h)&=t_\be(-2\pi i\La_0+2\pi i z+2\pi i \delta t) \\
           &=-2\pi i \tau(\La_0+\be-\frac{|\be|^2}{2}\delta)+2\pi i (z-\<z,\be\>\delta)+2\pi i t(\delta)\\
           &=-2\pi i \La_0 \tau+2\pi i (z-\tau\be)+2\pi i \delta(t-\<z,\be\>+\tau\frac{|\be|^2}{2}).
\end{align*}

\noindent Furthermore,
\begin{align*}
  (\al,\be,u)\cdot h&= -2\pi i \La_0 \tau+2\pi i (z+\a-\tau\be)+2\pi i \delta(t+\frac{u}{2\pi i}-\frac{\<\a,\be\>}{2}-\<z,\be\>+\tau\frac{|\be|^2}{2}).
\end{align*}
We are done.
\end{proof}

The action of groups $SL_2(\Z)$ and $N$ on $\h^*$ is compatible:
\begin{lem}
For $(\a,\beta,u)\in N$ and $h\in \h^*$ and $\begin{pmatrix}a & b \\c & d \end{pmatrix}\in SL_2(\Z)$, the following holds:
\begin{align}
\begin{pmatrix}a & b \\c & d \end{pmatrix}\cdot (\a,\beta,u)\cdot h=(a\al+b\beta,c\a+d\beta,u)\cdot\begin{pmatrix}a & b \\c & d \end{pmatrix}\cdot h.
\end{align}
\end{lem}

The metaplectic group is defined as
\[Mp_{2}(\Z):=\left\{        (A,j) \left| \begin{array}{c}   A\in SL_2(\Z),     \\
  j\;\; \text{is a holomorphic function in}\;\; \tau\in\mathbb{H}  \\
            \text{such that}\;\; j(\tau)^2=c\tau+d
\end{array}\right.\right\}.\]

\noindent Given a holomorphic function $F$ on $Y=\mathbb{H}\times \C\times \C$, one has the right action of $Mp_2(\Z)$ and $N$ on $F$:
\begin{align*}
   & F|_{(A,j)}(\tau,z,t):=\frac{1}{j(\tau)^{\ell}}\cdot F(A\cdot(\tau,z,t)),\\
   & F|_{(\a,\be,u)}(\tau,z,t):=F((\a,\be,u)\cdot(\tau,z,t)).
\end{align*}

Then one defines very important functions for $\al,\be\in \overline{\h}^*$:
\begin{align}
    F^{\al,\be}(\tau,z,t):=F((\al,\be,0)\cdot(\tau,z,t)),
\end{align}
namely
\begin{align*}
    F^{\a,\be}(\tau,z,t)=F(\tau,z+\a-\tau\be,t-\frac{\<\a,\be\>}{2}-\<\be,z\>+\frac{\tau}{2}|\be|^2).
\end{align*}
Modular transformation of these functions are given as follows:
\begin{lem}\label{tc1}\cite{wakimoto2001infinite}
Under the action of $(A,j)\in Mp_{2}(\Z)$,
 \begin{itemize}
     \item $(F|_{(A,j)})^{\a,\be}=F^{a\a+b\be,c\a+d\be}|_{(A,j)},$
     \item $F^{\a,\be}|_{(A,j)}=(F|_{(A,j)})^{a'\a+b'\be,c'\a+d'\be},$
 \end{itemize}

 where $A^{-1}=\begin{pmatrix}a' & b' \\c' & d' \end{pmatrix}$.
\end{lem}

 \subsection{Twisted characters}
 Now we consider the normalized character of $L(\lambda)$ at level $k$, $\chi_\lambda(h)=q^{m_{\La}}\tr_{L(\La)}e^{2\pi i h}$ evaluated at $h=2\pi i (-\tau d+ z+ tc)$, i.e., $\chi_{\lambda}(\tau,z,t)$.
 Then given $\a,\be\in \overline{\h}^*$, one has
 \begin{align}\label{sfct}\begin{split}
          \chi_{\lambda}^{\a,\be}(\tau,z,t)&=q^{m_{\lambda}}\tr_{L(\lambda)}e^{-2\pi i \tau d+2\pi i (z+\a-\tau\be)+2\pi i k(t-\frac{\<\a,\be\>}{2}-\<z,\be\>+\tau\frac{|\be|^2}{2})}\\
                               &=e^{2\pi i kt}\tr_{L(\lambda)}                                                           e^{2\pi i ((z+\a)+k(-\frac{\<\a,\be\>}{2}-\<z,\be\>))}q^{-\be+k\frac{|\be|^2}{2}}q^{L_{(0)}-\frac{k}{24}} \\
                               &=e^{2\pi i kt}\tr_{L(\lambda)}                                                           e^{2\pi i (z+\a-k\frac{\<\a,\be\>}{2}-k\<z,\be\>))}q^{L_{(0)}-\be+k\frac{|\be|^2}{2}-\frac{k}{24}}
 \end{split}
 \end{align}
The twisted character (\ref{sfct}) is very important. When $z=0$, multiplying (\ref{sfct}) with $\eta(\tau)^k$ produces preciously the theta function defined on vertex algebra by Miyamoto \cite{miyamoto2000modular}.
When $\a=0$, (\ref{sfct}) can be written as
\begin{align}\label{tc2}
    \begin{split}
        \chi_{\lambda}^{0,\be}(\tau,z,t)&=(\mathbf{y}e^{-2\pi i \<z,\be\>}q^{\frac{|\be|^2}{2}})^k\tr_{L(\La)} e^{2\pi i z}q^{-\be}                                                          q^{L_{(0)}-\frac{k}{24}}
    \end{split}
\end{align}
where $\mathbf{y}:=e^{2\pi i t}.$

\subsection{Twisted modules}\label{twm}
Let $g$ be the automorphism of $V$ of finite order $T$. We first recall the definition of $g$-twisted module:
\begin{dfn}
An (ordinary) $g$-twisted $V$-module is a $\mathbb{C}$-linear space $M$ equipped with a linear map \begin{align*}
    & V\rightarrow {\rm End}(V)[[z^{\frac{1}{T}},z^{-\frac{1}{T}}]],\\
    & v\mapsto Y_{M}(v,z)=\displaystyle \sum_{n\in \mathbb{Q}}v_{(n)}z^{-n-1}
\end{align*} satisfying
\begin{itemize}
    \item For $v,w\in M$, $v_{(m)}w=0$ if $m$ is large enough.
    \item $Y_{M}(\mathbf{1},z)=id_{V}.$
    \item For $v\in V^{r}=\left\{v\in V|gv=e^{\frac{2\pi ir}{T}}v\right\}$, and $0\leq r\leq T-1$ \[Y_{M}(v,z)=\displaystyle\sum_{n\in \frac{r}{T}+\mathbb{Z}} v_{(n)}z^{-n-1}.\]
    \item\textbf{(Jacobi identity)} For $u\in V^{r}$
    \begin{align*}&z_{0}^{-1}\delta(\frac{z_{1}-z_{2}}{z_{0}})Y_{M}(v,z_{1})Y_{M}(w,z_{2})-(-1)^{|v||w|}z_{0}^{-1}\delta(\frac{z_{0}-z_{1}}{-z_{0}})Y_{M}(v,z_{2})Y_{M}(u,z_{1})\\&=z_{2}^{-1}(\frac{z_{1}-z_{0}}{z_{2}})^{-\frac{r}{T}}\delta(\frac{z_{1}-z_{0}}{z_{2}})Y_{M}(Y(u,z_{0})v,z_{2}).\end{align*}
    \item\textbf{(grading)} \begin{itemize}
        \item $M=\oplus_{\lambda\in\mathbb{C}}M_{\lambda},$ where $M_{\lambda}=\left\{w\in M|L_{(0)}w=\lambda w\right\}$,
        \item $M_{\lambda}$ is finite dimensional,
        \item for a fixed $\lambda$, $M_{\frac{n}{T}+\lambda}=0$ for all small enough integers $n$.
    \end{itemize}
\end{itemize}
\end{dfn}
\begin{rem}
If the grading condition is dropped, one calls it the weak $g$-twisted module.
\end{rem}

 Now we review some basic facts about twisted modules under the vertex algebra setting. Let $(V,\omega)$ be a $\Z$-graded conformal vertex algebra. Let $v\in V_1$ be an even vector satisfying the Heisenberg $\lambda$-bracket relation
\begin{align}
    \begin{split}
        [v_{\lambda}v]=k\lambda, \quad [\omega_{\lambda}v]=(T+\lambda)v.
    \end{split}
\end{align}
Suppose further that $v_{(0)}$ acts semisimply on $V$ such that the eigenvalues of $v_{(0)}$ belong to $\frac1T\Z$.  Li's $\Delta$-operator is defined as
\begin{align}
    \Delta(z):=z^{v_{(0)}}\exp(\sum_{n=1}^{\infty}\frac{v_{(n)}}{-n}(-z)^{-n}).
\end{align}
 When $g(v)=v$, one can obtain an $ge^{2\pi i v_{(0)}}$-twisted module from a $g$-twisted $V$-module $M$ by using Li's $\Delta$-operator 
 as the following

\begin{prop}\cite{li1996local}\label{li1996local}
 $(M,Y_{M}(\Delta(z)\cdot,z))$ is a weak $ge^{2\pi i v_{(0)}}$-twisted module.
\end{prop}
\noindent When $g=\text{id}$, $e^{2\pi i v_{(0)}}$ is an automorphism of $V$ of order $T$. Then $(M,Y_{M}(\Delta(z)\cdot,z))$ is a $e^{2\pi i v_{(0)}}$-twisted module.  Let $v'\in V_1$ satisfy $v_{(0)}v'=0$. Define
\begin{align}
\begin{split}
   & \hat{L}_{(0)}:=\text{Res}_z zY_{M}(\Delta(z)\omega,z), \\
   & \hat{v'}_{(0)}:=\text{Res}_z Y_{M}(\Delta(z)v',z).
 \end{split}
\end{align}

\noindent By direct calculation, one has
\begin{align*}
  \hat{L}_{(0)}=L_{(0)}+v_{(0)}+\frac12\<v,v\>,\quad \hat{v'}_{(0)}=v'_{(0)}+\<v,v'\>.
\end{align*}
The nomalized character of weak $V$-module $M$ is defined as
\begin{align}
    \text{ch}(M)(q,z)=\tr_{M}e^{2\pi i z v'_{(0)}}q^{L_{(0)}-\frac{c}{24}}.
\end{align}
Then the normalized character of weak $e^{2\pi i v_{(0)}}$-twisted module $\hat{M}=(M,Y_M(\Delta(z)\cdot,z))$ is
\begin{align}\label{sfct2}
  \text{ch}(\hat{M})(q,z)=\tr_{M}e^{2\pi i z (v'_{(0)}+k\<v,v'\>)}q^{L_{(0)}+v_{(0)}+\frac12k\<v,v\>-\frac{c}{24}}.
\end{align}
\begin{rem}
 In the affine vertex algebra case, let $M=L(\La)$, $v'=z$ and $v=-\be$. One has
 \[\mathrm{ch}(\hat{M})(q,z)=\chi_{\La}^{0,-\be}(\tau,z,0),\] where $q=e^{2\pi i \tau}.$ It is also worth noting that the central charge of the conformal element $\omega+L_{(-1)}v$  is $c-12\<v,v\>$. Also $(\omega+L_{(-1)}v)_{(0)}=L_{(0)}+v_{(0)}$. Thus, \begin{align}\label{sfct3}\tr_{M}e^{2\pi i z v'_{(0)}}q^{L_{(0)}+v_{(0)}+\frac12k\<v,v\>-\frac{c}{24}}\end{align} can be understood as the trace function on $M$ with respect to the new conformal vector, $\omega+L_{(-1)}v$. Careful readers may notice the difference between (\ref{sfct2}) and (\ref{sfct3}); this is because the first one changes the module structure and the second one changes the conformal structure.
\end{rem}

\subsection{Spectral Flow Automorphsim}
Let us give a brief introduction to the spectral flow automorphisms  (see \cite{Ridout:2008nh,Creutzig:2012sd,Creutzig:2013yca} for details). Let  $\hat{\mathfrak{g}}$
be the (non-twisted) affine Lie algebra associated with the semisimple complex finite  Lie algebra
($\mathfrak{g}$,$(~|~))$. Let $\Delta$ be the set of roots of $\mathfrak{g}$. Let $\alpha\in \Delta$ denote a root of $\mathfrak{g}$ with root vector $e^{\alpha}$ and coroot $\alpha^{\vee}$, and let $W$ be the Weyl group of finite Lie algebra $\mathfrak{g}$.  Each element $w\in W$ permutes the roots and  induces an automorphism of $\mathfrak{g}$ via $w(e^{\alpha})=e^{w(\alpha)}$, this action can be generalized to an affine Lie algebra $\hat{\mathfrak{g}}$ as follows, the corresponding root vector of real roots are $e^{\alpha}_{n}$, and the roots vector corresponding to imaginary roots are denoted by $h^{i}_{n}$, one can associate $h^{i}$ with the simple coroot $\alpha^{\vee}_{i}$ of $\mathfrak{g}$. Let $W\subset GL(\mathfrak{h}^{*})$ be the Weyl group of $\mathfrak{g}$
generated by the reflections $s_{\alpha}$
with $\alpha\in \hat{\Delta}^{\text{re}}$,
where
$s_{\alpha}(\lambda)=\lambda-\langle \lambda,\alpha^{\vee}\rangle\alpha$
for $\lambda\in\mathfrak{h}^{*}$.
Let the affine Weyl group be
$\widehat{W}=W\ltimes \bar{Q}^{\vee}$. The coroot lattice acts on the roots of $\hat{\mathfrak{g}}$ by translation in the imaginary direction. Then each simple coroot $\alpha^{\vee}_{i}$ of $\mathfrak{g}$ defines an independent transformation $\tau_{i}$ on the root vector of an affine Lie algebra $\hat{\mathfrak{g}}$ via,
\begin{align}
\tau_{i}(e_{n}^{\alpha})=e^{\alpha}_{n-\langle\alpha,\alpha^{\vee}_{i}\rangle} \quad (n\in\mathbb{Z}), \qquad \tau_{i}(h^{j}_{n})=h^{j}_{n} \quad (n\neq0)
\end{align}
We finally obtain the spectral flow automorphisms $\tau_{i}$ act on generators of the affine Lie algebra
$\hat{\mathfrak{g}}$ and Virasoro element $L_{0}$ as follows,
\begin{align}
	\tau_{i}(e^{\alpha}_{n})&=e^{\alpha}_{n-\langle\alpha,\alpha^{\vee}_{i}\rangle}, \quad \tau_{i}(h^{j}_{n})=h^{j}_{n}-(\alpha^{\vee}_{i}|\alpha^{\vee}_{j})K\\
	\tau_{i}(K)&=K, \quad \tau_{i}(L_{0})=L_{0}-h^{i}_{0}+\frac{2}{(\alpha_{i}|\alpha_{i})}K
\end{align}
For the $\hat{\mathfrak{g}}=A^{(1)}_{1}$,
the spectral flow automorphism $\sigma$ which may be regarded as a square root of the affine Weyl translation by the simple coroot $\alpha^{\vee}_{i}$ of finite Lie algebra $\mathfrak{g}$, the powers of $\sigma$ acts as follows,
\begin{align}
\sigma^{\ell}(e_{n})&=e_{n-\ell}, \quad \sigma^{\ell}(h_{n})=h_{n}-\ell\delta_{n,0}K, \quad \sigma^{\ell}(f_{n})=f_{n+\ell}\\
\sigma^{\ell}(K)&=K, \quad \sigma^{\ell}(L_{0})=L_{0}-\frac{1}{2}\ell h_{0}+\frac{1}{4}\ell^{2}K
\end{align}
One can use spectral flow automorphsim to modify the action of $A^{(1)}_{1}$ on any module $M$, thereby obtaining new modules $\sigma^{*}(M)$. Explicitly, the modified algebra action defining these new modules is given by,
\begin{equation}
	X\cdot \sigma^{*}|v\rangle=\sigma^{*}(\sigma^{-1}(X)|v\rangle), \quad (X\in \mathfrak{sl}_{2}). \end{equation}
For example, if $|\lambda,\Delta\rangle$ is a vector of weight $\lambda$ and conformal dimension $\Delta$, then the vector $(\sigma^{\ell})^{*}|\lambda,\Delta\rangle\in (\sigma^{\ell})^{*}(M)$, the weight and conformal dimension of this vector becomes,
\begin{align}\label{spectral1}
	h_{0}(\sigma^{\ell})^{*}|\lambda,\Delta\rangle=(\lambda+\ell K)(\sigma^{\ell})^{*}|\lambda,\Delta\rangle
\end{align}
\begin{align}\label{spectral2}
L_{0}(\sigma^{\ell})^{*}|\lambda,\Delta\rangle=(L_{0}+\frac{1}{2}\ell h_{0}+\frac{1}{4}\ell^{2}K)(\sigma^{\ell})^{*}|\lambda,\Delta\rangle
\end{align}
(From now on, we denote the new module as $\sigma^{\ell}(M)$). In order to obtain the character of new module $\sigma^{\ell}(M)$, we need to compute the character of module $M$. The normalized character of irreducible highest weight $\hat{\mathfrak{g}}$ module $M$ at level $k$ is defined as,
\begin{equation}
\text{ch}[M](\mathbf{y},\mathbf{z},q)=\text{tr}_{M}\,\mathbf{y}^{k}\mathbf{z}^{h_{0}}q^{L_{0}-c/24}
\end{equation}
where $\mathbf{y}=e^{2\pi i t}$, $\mathbf{z}=e^{2\pi i z}$ and $q=e^{2\pi i \tau}$.

The character of new module $\sigma^{\ell}(M)$ can be written in terms of the character of module $M$ as follows,
\begin{equation}\label{twitransf}
\text{ch}[\sigma^{\ell}(M)](\mathbf{y},\mathbf{z},q)=\text{ch}[M]\left(\mathbf{y}\mathbf{z}^{\ell}q^{\ell^{2}/4},\mathbf{z}q^{\ell/2},q\right),
\end{equation}
One can check that the character of module $M$ satisfies the following relation,
\begin{equation}
\text{ch}[\sigma^{\ell+\ell'}(M)](\mathbf{y},\mathbf{z},q)=\text{ch}[\sigma^{\ell}\circ\sigma^{\ell'}(M)](\mathbf{y},\mathbf{z},q)=\text{ch}[M]\left(\mathbf{y}\mathbf{z}^{\ell+\ell'}q^{(\ell+\ell')^{2}/4},\mathbf{z}q^{(\ell+\ell')/2},q\right).
\end{equation}
\subsection{Li's delta operator, spectral flow and twisted modules}
Let $M$ be a highest weight $\hat{\mathfrak{g}}$-module or its contragredient module. 
Let $v\in \mathfrak{h}_{\mathbb{R}}^*$. We have
\begin{align*}
   & \Delta(z) e_{-1}^{\al}\mathbf{1}=z^{\la v,\al \ra}e_{-1}\mathbf{1}, \\
   & \Delta(z) h_{-1}^j\mathbf{1}=h_{-1}^j\mathbf{1}+\la v,h^{j} \ra kz^{-1}\mathbf{1}
\end{align*}  
Thus, it induces the spectral flow $\tau$:
\begin{align*}
   & \tau(e_{n}^{\al})=e_{n+\la v, \al \ra}^\al \\
   & \tau(h_n^{j})=h_n^{j}+\la v,h^{j} \ra k \delta_{n,0}.
\end{align*}
If $\la v, \al \ra\in \frac{1}{T}\mathbb{Z}$ for all roots $\al$, according to Proposition \ref{li1996local}, $(M,Y_{M}(\Delta(z)\cdot,z))$ is a weak $e^{2\pi i v_{(0)}}$ -twisted ($\mathbb{Z}_{T}$-twisted) module. But the $\Z_{T}$-twisted modules are not necessarily coming from spectral flow.

In the future, we shall call $(M,Y_{M}(\Delta(z)\cdot,z))$ the spectral flowed module or $e^{2\pi i v_{(0)}}$-twisted module. 
\begin{prop}\label{admissible}
	For $A^{(1)}_{1}$ at admissible  level $k=-2+\frac{v}{u}$, let $\lambda(i,j,v,u)$ be the Dynkin label for the finite Lie algebra  $\mathfrak{sl}_{2}$ of an irreducible highest weight module and $\Delta(i,j,v,u)$ its conformal weight determined by the Sugawara construction. They  become $\sigma^{-\frac{1}{2}}(\lambda(i,j,v,u))$ and $\sigma^{-\frac{1}{2}}(\Delta(i,j,v,u))$ after taking $\ell=-\frac{1}{2}$ spectral flow action on $\lambda(i,j,v,u)$ and $\Delta(i,j,v,u)$ as follows,
	\begin{align*}
		&\sigma^{-\frac{1}{2}}(\lambda(i,j,v,u))=i-1-(k+2)s-\frac{k}{2}, \quad (i=1,\cdots,v-1, j=0,\cdots, u-1) \\ 
		&\sigma^{-\frac{1}{2}}(\Delta(i,j,v,u))=\frac{1}{16}\left(4+k-4i+4(2+k)j+\frac{4(-1+(i-(2+k)j)^{2})}{2+k}\right), \quad (i=1,\cdots,v-1, j=0,\cdots, u-1) \\ 
	\end{align*}
\end{prop}
\begin{proof}
	The statement follows from (\ref{spectral1}), (\ref{spectral2}) and theorem (3.5.3) in \cite{adamovic1995vertex}.
\end{proof}
One can compute the $\mathbb{Z}_{2}$-twisted Zhu's algebra $A_{\sigma}(L_{k}(\mathfrak{sl}_{2}))$ at  admissible levels $k=-2+\frac{v}{u}$ (The proof will be given in Proposition \ref{isozhu}), and find that all irreducible $\mathbb{Z}_{2}$-twisted modules of $L_{k}(\mathfrak{sl}_{2})$ from $\ell=-\frac{1}{2}$ spectral flow on the untwisted modules in category $\mathcal{O}$.
\subsection{Half-integer spectral flow for $A_{1}^{(1)}$ at boundary admissible level}\label{bal1}
For $A_{1}^{(1)}$,  $h^{\vee}=2$,  the boundary admissible levels are $k=-2+\frac{2}{u}$, where $u=2n+1$ is a positive odd integer. All admissible weights are,
\begin{equation}
\Lambda_{k,j}:=t_{-\frac{j}{2}}\cdot(k\Lambda_{0})=\left(k+\frac{2j}{u}\right)\Lambda_{0}-\frac{2j}{u}\Lambda_{1}, \quad j=0,1,\cdots,u-1,
\end{equation}
where $\Lambda_{0}$ and $\Lambda_{1}$ are fundamental weights of affine Lie algebra $A^{(1)}_{1}$. All  characters of irreducible modules can be written in terms of Jacobi theta function $\theta_{1}(\mathbf{z};q)$ as follows,
\begin{equation}	\text{ch}[L(\Lambda_{k,j})](\mathbf{y},\mathbf{z},q)=\mathbf{y}^{k}\mathbf{z}^{-\frac{2j}{u}}q^{\frac{j^{2}}{2u}}\frac{\theta_{1}(\mathbf{z}^{2}q^{-j};q^{u})}{\theta_{1}(\mathbf{z}^{2};q)}, \quad (j=0,1,\dots,u-1)
\end{equation}
We should emphasize that the vaccum character $(j=0)$ of these VOAs denoted by  $
L_{\frac{-4n}{2n+1}}(\mathfrak{sl}_{2})$ coincide with superconformal indices of the $4d$ supersymmetric gauge theories which called $(A_{1},D_{2n+1})$ theories. Especially,
when $u=3$,  level $k=-4/3 $, the  $L_{-\frac{4}{3}}(\mathfrak{sl}_{2})$ coincides with the $\mathfrak{a}_{1}$ DC series of simple Lie algebras.

One can generalize the integer flow parameter $\ell$ to half integer, (\ref{twitransf}) stays the same.
For the $A_{1}^{(1)}$ at boundary admissible levels $k=-2+\frac{2}{u}$, using (\ref{twitransf}), all characters of these twisted modules take the following form,
\begin{equation}
	\text{ch}[\sigma^{-\frac{1}{2}}({L({\Lambda_{k,j}}}))](\mathbf{y},\mathbf{z},q)=(\mathbf{y} \mathbf{z}^{-\frac{1}{2}}q^{\frac{1}{16}})^{k}(\mathbf{z} q^{-\frac{1}{4}})^{-\frac{2j}{u}}q^{\frac{j^{2}}{2u}}\frac{\theta_{1}(\mathbf{z}^{2}q^{-j-\frac{1}{2}};q^{u})}{\theta_{1}(\mathbf{z}^{2}q^{-\frac{1}{2}};q)}, \quad (j=0,1,\dots,u-1)
\end{equation}
\label{prop2.7}
Now, we use $\lambda(j,u)$ instead of  $\lambda(1,j,2,u)$ (Similar to $\Delta(j,u)$), and get,
\begin{align*}
\sigma^{-\frac{1}{2}}(\lambda(j,u))=\frac{u-2j-1}{u}, \quad \sigma^{-\frac{1}{2}}(\Delta(j,u))=\frac{1+4j(1+j-u)-u}{8u}, \quad (j=0,1,\dots u-1)
\end{align*}
In particular, we observe that the $\sigma^{-\frac{1}{2}}(\lambda(j,u))$ and $\sigma^{-\frac{1}{2}}(\Delta(j,u))$ satisfy the following relations respectively,
\begin{equation}
	\sigma^{-\frac{1}{2}}(\lambda(j,u))=-\sigma^{-\frac{1}{2}}(\lambda(u-j-1,u)), \quad \sigma^{-\frac{1}{2}}(\Delta(j,u))=\sigma^{-\frac{1}{2}}(\Delta(u-1-j,u))
\end{equation}
Fix a boundary admissible level $k$, all $\sigma^{-\frac{1}{2}}({L({\Lambda_{k,j}}}))$ are ordinary modules. Then the normalized characters satisfy a modular linear differential equation (Section \ref{section5}). 
We shall use the fact that $A_{\hat{\sigma}}(L_{k}(\mathfrak{sl}_{2}))$ is semi-simple to prove semisimplicity of $\mathbb{Z}_{2}$-twisted modules of $L_{k}(\mathfrak{sl}_{2})$.
\begin{lem}
Let $\sigma^{-\f12}(L(\Delta_{1},\lambda_{1}))$ and $\sigma^{-\f12}(L(\Delta_{2},\lambda_{2}))$ be irreducible $\Z_2$-twisted modules of $L_{k}(\mathfrak{sl}_{2})$.  Suppose that there is a nontrivial extension of
\begin{align*}
0\longrightarrow \sigma^{-\f12}(L(\Delta_{1},\lambda_{1}))\overset{\iota}{\longrightarrow} M\overset{\pi}{\longrightarrow} \sigma^{-\f12}(L(\Delta_{2},\lambda_{2}))\longrightarrow 0
\end{align*}
Then $L_0$ acts locally finitely on $M$.
\begin{proof}
Since the Zhu's algebra $A_{\hat{\sigma}}(L_{k}(\mathfrak{sl}_{2}))\cong \mathbb{C}[h]/\langle I \rangle$ is a commutative algebra and semisimple. According to characters  $\text{ch}[\sigma^{-\frac{1}{2}}({L(\Lambda_{k,j}){}})]$ of all twisted modules. We have $\sigma^{-\f12}(L(\Delta_{1},\lambda_{1}))$ and $\sigma^{-\f12}(L(\Delta_{2},\lambda_{2}))$ are $L_{0}$-diagonalizable and the $L_{0}$-finite dimensional. The proof is similar to \cite{Fasquel:2020iqf} (Lemma 7.3).  Finally, we have that, $m\in M$ belongs to some $L_0$-stable finite dimensional vector subspace of
\begin{align*}
\bigoplus_{i=1}^r\ker(L_0-\nu_i Id)\oplus\bigoplus_{j=1}^s\ker(L_0-\mu_j Id).
\end{align*}
where $\nu_{1},\cdots,\nu_{r}$ and $\mu_{1},\cdots,\mu_{s}$ are eigenvalues of some eigenvectors $v_{1},\cdots,v_{r}\in \sigma^{-\f12}(L(\Delta_{1},\lambda_{1}))$ and $w_{1},\cdots,w_{s}\in \sigma^{-\f12}(L(\Delta_{2},\lambda_{2}))$ respectively, then the assertion follows.
\end{proof}
\end{lem}
Then we can prove the following theorem which is Theorem \ref{thm:mainThm2} in the introduction.
\begin{thm}
The category of $\Z/2\Z$-twisted modules of the vertex operator algebra  $L_{k}(\mathfrak{sl}_{2})$ at boundary level $k$ is semi-simple.
\begin{proof}
For any distinct conformal dimension $\sigma^{-\frac{1}{2}}(\Delta_{m}), \sigma^{-\frac{1}{2}}(\Delta_{n})\in \sigma^{-\frac{1}{2}}(\Delta(j,u))$ which correspond to two distinct simple modules. We have $\sigma^{-\frac{1}{2}}(\Delta_{m})\neq\sigma^{-\frac{1}{2}}(\Delta_{n}) \, (\mathrm{mod}\,\mathbb{Z})$. It
is known \cite{Fasquel:2020iqf} (Lemma 7.4) that if there exists a nontrivial extension of two distinct simple $\mathbb{Z}_{2}$-twisted $L_{k}(\mathfrak{sl}_{2})$-modules
\begin{align*}
0\longrightarrow \sigma^{-\f12}(L(\Delta_{m},\lambda_{m}))\overset{\iota}{\longrightarrow} M\overset{\pi}{\longrightarrow} \sigma^{-\f12}(L(\Delta_{n},\lambda_{n}))\longrightarrow 0
\end{align*}
then $\sigma^{-\frac{1}{2}}(\Delta_{m})$ and $\sigma^{-\frac{1}{2}}(\Delta_{n})$ coincide modulo $\mathbb{Z}$. Then, we have $\mathrm{Ext}^{1}[\sigma^{-\f12}(L(\Delta_{m},\lambda_{m})),\sigma^{-\f12}(L(\Delta_{n},\lambda_{n}))]=0$ for $\sigma^{-\frac{1}{2}}(\Delta_{m})\neq\sigma^{-\frac{1}{2}}(\Delta_{n})$. For the simple modules which have the same conformal dimension  $\sigma^{-\frac{1}{2}}(\Delta_{m})=\sigma^{-\frac{1}{2}}(\Delta_{n})$.
One can consider the top space of these modules, the exact sequence
\begin{align*}
0\longrightarrow \sigma^{-\f12}(L(\Delta_{m},\lambda_{m}))_{\mathrm{top}}\overset{\iota}{\longrightarrow} M_{\mathrm{top}}\overset{\pi}{\longrightarrow} \sigma^{-\f12}(L(\Delta_{n},\lambda_{n}))_{\mathrm{top}}\longrightarrow 0
\end{align*}
has non-trivial extension of $A_{\hat{\sigma}}(L_{k}(\mathfrak{sl}_{2}))$ modules, This contradicts the fact that Zhu’s algebra $A_{\hat{\sigma}}(L_{k}(\mathfrak{sl}_{2})$ is semisimple. So, there is no nontrivial extension between two distinct simple modules and nontrivial self-extension. This completes the proof.
\end{proof}
\end{thm}

\subsection{Modular transformation}
Applying (\ref{tc1}) to (\ref{sfct}), one has the following $S$-transformation for twisted characters:
\begin{align}
    \begin{split}
      \chi_{\La}^{\a,\beta}(\tau,z,t)|_{S}=(\chi_{\La}(\tau,z,t)|_{S})^{\be,-\a},
    \end{split}
\end{align}
where $S=\begin{pmatrix}0 & -1 \\1 & 0 \end{pmatrix}$.

\begin{ex}{(\textbf{Boundary admissible level})}
We follow the same notations in Section \ref{bal1}. Let $\text{Spec}(T)$ be the eigenvalues of $T$. First, note that $\text{Spec}( (\epsilon  h_{1})_{(0)})\in 2\epsilon \Z$, $(\epsilon\in\mathbb{Q})$. \noindent The twisted character (\ref{tc2}) evaluated at $h=2\pi i (-\tau d+ \al_1 z+ct)$ is
\begin{align}\label{bal3}
\begin{split}
    \chi_{\La}^{0,-\frac{\ell}{2}\a_{1}}(\tau,z,t)&=(\mathbf{y}e^{-2\pi i \<\al_1 z,-\frac{\ell}{2}\a_{1}\>}q^{\frac{|-\frac{\ell}{2}\a_{1}|^2}{2}})^m\tr_{L(\La)} e^{2\pi i \al_1 z}q^{\frac{\ell}{2}\al_{1}}                                                          q^{L_{(0)}-\frac{m}{24}}\\
                                                &=(\mathbf{y}\mathbf{z}^{\ell}q^{\frac{\ell^2}{4}})^m\tr_{L(\La)} (\mathbf{z}q^{\frac{\ell}{2}})^{h_{0}}                                                          q^{L_{(0)}-\frac{m}{24}},
\end{split}
\end{align}
where we identity $(h_{1})_{(0)}$ with $\alpha_{1}$. (\ref{bal3}) is the same as (\ref{twitransf}).

We split the discussion into two cases:
\begin{itemize}
    \item $\text{Spec}( (\epsilon  h_{1})_{0})\in \Z$, i.e., $\epsilon=\frac{\ell}{2},$ $\ell\in \mathbb{Z}$. Let $\sigma$ be spectral flow automorphism on $M$ (see \cite{creutzig2013modular} for more details). In this case, (\ref{bal3}) is the character of $\sigma^{\ell}(M)$.
    \item $\text{Spec}( (\epsilon h_{1})_{0})\in \frac1T\Z,$ $T\in \Z_{>0}$, i.e., $\epsilon=\frac{\ell}{2}$, $\ell\in \frac1T\Z,$ $T\in \Z_{>0}$. (\ref{bal3}) now becomes the character of the $e^{2\pi i (\epsilon  h_{1})_{(0)}}$-twisted module $\hat{M}$ introduced in Section \ref{twm}, where $v=\epsilon h_{1}$.  The $S$-transformation of $\chi_{\lambda_{m,k}}^{0,-\frac{\ell}{2}\a_{1}}(\tau,z,t)$ is
    \begin{align}
    \begin{split}
      \chi_{\lambda_{m,k}}^{0,-\frac{\ell}{2}\a_{1}}(\tau,z,t)|_{S}&= (\chi_{\lambda_{m,k}}(\tau,z,t)|_{S})^{-\frac{\ell}{2}\a_{1},0}\\
                                                                 &=\displaystyle \sum_{\lambda_{m,k'}\in P_{m}}a_{\lambda_{m,k},\lambda_{m,k'}}(\chi_{\lambda_{m,k'}}^{-\frac{\ell}{2}\a_{1},0}(\tau,z,t))
      \end{split}
    \end{align}
    where \[a_{\lambda_{m,k},\lambda_{m,k'}}=(-1)^{k+k'}e^{-\frac{2\pi i k k'}{u}}\frac{1}{\sqrt{u}}\sin\frac{u\pi}{2}, \]and
    \begin{align}
    \begin{split}
      \chi_{\lambda_{m,k'}}^{-\frac{\ell}{2}\a_{1},0}(\tau,z,t)=\chi_{\lambda_{m,k'}}(\tau,z-\frac{\ell}{2},t).
     \end{split}
    \end{align}
\end{itemize}
As one can see, the $S$ transformation of the character of  $e^{2\pi i (\sigma h_{1})_{(0)}}$-twisted module $L(\lambda_{m,k})$ can be written as linear combination of the characters of untwisted modules with the same $S$-matrix as in untwisted case.
\end{ex}

\section{Twisted Zhu's Bimodule of highest weight modules}\label{section3}

The fusion rules among admissible representations in category $\mathcal{O}$ were studied by many authors (\cite{dong1997vertex,Ridout:2008nh,Ridout:2010qx,Creutzig:2012sd,Creutzig:2013yca,Gaberdiel:2001ny}, and etc.). They found that the Verlinde formula is no longer true in the case of $L_{k}(\mathfrak{g})$, where $k$ is the admissible level and $\mathfrak{g}$ is a simple Lie algebra, since the negative coefficient would appear. The method used by Dong-Li-Mason is to apply Frenkel-Zhu's bimodule theorem \cite{FrenkelZhu}, while the other methods involve the use of characters and some machinery from physics \cite{Awata:1992sm,Gaberdiel:2001ny,Lesage:2002ch}. In this section, we shall compute the twisted Zhu's bimodules recently introduced in \cite{zhu2022bimodues}; in particular, the twisted Zhu's algebra is a special case, then we use the twisted Zhu's algebra and bimodules to classify the twisted modules and also calculate the fusion rules among the twisted modules.

Let $M^1, M^2, M^3$ be $g_1, g_2, g_3$-twisted modules, respectively. We consider the fusion rules $N^3_{1\;2}$, where $1,2,3$ represent modules $M^1, M^2, M^3$. Recall the following definition \cite{zhu2022bimodues} of $A_{g_2}$-bimodule $A_{g_1g_2,g_2}(M^1)$. Denote the remainder of $r\in \mathbb{N}$ divided by $T$ by $[r]$.

For homogeneous element $u\in V$ and $w_1\in M^1$, one defines
\[u\circ_{g_1g_2,g_2}w_1=\res{z}{\f{(1+z)^{\wt u-1+\delta(j_2)+\f{j_2}{T}}}{z^{1+\delta(j_1,j_2)-\f{j_1}{T}}}}\]
where $u\in V^{(j_1,j_2)}$ and
\begin{equation}
\delta(j_1,j_2)=\begin{cases}
1,& j_2=0\\
1,& j_2\neq 0,\; j_1+j_2\geq T\\
0,& j_2\neq 0,\; j_1+j_2<T.
\end{cases}
\end{equation}
Let $O'_{g_1g_2,g_2}(M^1)$ be the subspace of $M^1$ spanned by all $u\circ_{g_1g_2,g_2}w_1$. One defines $A_{g_1g_2,g_2}(M^1)=M^1/O'_{g_1g_2,g_2}(M^1)$.

Now we recall \cite{zhu2022bimodues} the $A_{g_1g_2}(V)$-$A_{g_2}(V)$-bimodule on $A_{g_1g_2,g_2}(M)$. For homogeneous $u\in V$ and $w\in M$, the left and right bimodule actions are defined as
\begin{equation}
  {u*_{g_1g_2,g_2}w}=
  \begin{cases}
    {\res{z}{Y_{M}(u,z)w\f{(1+z)^{\wt u-1+\delta(j_2)+\f{j_2}{T}}}{z^{1-\f{j_1}{T}}}}}  &  {j_1+j_2\equiv0\;(\text{mod}\;\;T)}\\
    {0}  &  \text{otherwise},
  \end{cases}
\end{equation}
and
\begin{equation}\label{multiplication}
  w*_{g_2,g_1g_2}u=
  \begin{cases}
    \res{z}{Y_{M}(u,z)w\f{(1+z)^{\wt u-1}}{z^{1-\f{j_1}{T}}}}  &  {j_2=0}\\
    {0},  &  \text{otherwise}.
  \end{cases}
\end{equation}

 In particular, when $g_1=id$, $g_2=g$, the $A_{g,g}(V)$ with the multiplication given by (\ref{multiplication}) is the same as the $g$-twisted Zhu's algebra, $A_g(V)$ defined in \cite{dong2000modular}.

\begin{conj}\cite{zhu2022bimodues}
One has:
\[\dim \hom_{A_{g_1g_2}(V)}(A_{g_1g_2,g_2}(M^{1})\otimes M^{2}(0),M^{3}(0))=N^{3}_{1\; 2}.\]	
\end{conj}
\begin{rem}
When $g_1=g_2=id$, this Conjecture was proved by Frenkel and Zhu in \cite{FrenkelZhu}.	
\end{rem}

 Let $e,f,h$ be the basis of $\mathfrak{sl}_2$. Define $\hat{\sigma}=e^{\f{\pi ih_{(0)}}{2}}$.  One can obtain  irreducible $\hat{\sigma}$-twisted ($\mathbb{Z}_{2}$-twisted) modules of $L_{\mathfrak{sl}_2}(\ell,0)$ in category $\mathcal{O}$ by using Li's Delta operator:
\[\Delta_{-\f12}(z)=z^{\f14h_{(0)}}
\exp\left(\sum_{n=1}^{\infty}\f{\f14h_{(n)}}{-n}(-z)^{-n}\right),\]
 i.e., $\sigma^{-\f12}(L_{\s_2}(\ell,j))$, where we define \[(\sigma^{-\f12}(M),Y^{-\f12}_M(\cdot,z))
:=(M,Y_M(\Delta_{-\f12}(z)\cdot,z)). \] 

Now we compute the twisted Zhu's bimodule in the case where $M^1$ is untwisted and
$M^2,M^3$ are $\hat{\sigma}$-twisted modules.
We have
\begin{equation}
u\circ_{\hat{\sigma},\hat{\sigma}}w=
\res{z}{\f{(1+z)^{\wt u-1+\delta(j_2)+\f{j_2}{2}}}{z^{1+\delta(j_2)}}}Y_M(u,z)w,
\end{equation}
where $u\in V^{(0,j_2)}$.
If $u\in V_1^{(0,1)}$, one has
\begin{align}\label{fu1}
\begin{split}
&u\circ_{\hat{\sigma}, \hat{\sigma}}w\\
&=\res{z}{\f{(1+z)^{\f12}}{z^{m+1}}}Y_M(u_{(-1)}\mathbf{1},z)w\\
&=u(-m-1)w+\f12u(-m)w-\f18u(-m+1)w+\f{1}{16}u(-m+2)w+\cdots\in O_{\hat{\sigma},\hat{\sigma}}(M),\end{split}\end{align}
where $m\geq0$.

In our case, we have the following bimodule action \cite{dong1997vertex}
\begin{prop}
  The $A_{\hat{\sigma}} (V)$-bimodule $A_{\hat{\sigma},\hat{\sigma}}(\sigma^{-\f12}(M))$ is isomorphic to $\C[x,y]$ with bimodule action as follows:
  \begin{align}\label{biact}
    x*f(x,y)=(x+j-2y\f{\partial}{\partial y})f(x,y), \quad f(x,y)*x=xf(x,y)
  \end{align}
  for any $f(x,y)\in \C[x,y]$, where $h_{(0)}v=jv$.
\end{prop}
\begin{proof}
  By Definition, we have
\begin{align*}
    h_{(-1)}*(h_{(-1)}^mf_{(0)}^nv)&=(h_{(-1)}+h_{(0)})h_{(-1)}^mf_{(0)}^nv\\
                                 &=(h_{(-1)}+j-2n)h_{(-1)}^mf_{(0)}^nv
\end{align*}
and
\begin{align*}
  (h_{(-1)}^mf_{(0)}^n)*h_{(-1)}=h_{(-1)}(h_{(-1)}^mf_{(0)}^n)=h_{(-1)}^{m+1}f_{(0)}^nv.
\end{align*}
The Proposition follows immediately if we set
$x=h_{(-1)}+O(M(\ell,j)), y=f(0)+O(M(\ell,j))$.
\end{proof}

\subsection{Examples}
Let $g$ be an automorphism of the universal affine vertex algebra  $V^k(\mathfrak{g})$ and denote by $\mathfrak{g}^0$ the fixed point subalgebra of $\mathfrak{g}$ under the action of $g$.
The $g$-twisted Zhu's algebra of  $V^k(\mathfrak{g})$ is the universal enveloping algebra of $\mathfrak{g}^0$, $U(\mathfrak{g}^0)$, via the map
\cite{yang2017twisted}
\begin{align*}
    & F: A_{g}(V^k(\mathfrak{g}))\mapsto U(\mathfrak{g}^0)\\
    & F([x^1_{(-n_1-1)}x^2_{(-n_2-1)}\cdots x^m_{(-n_m-1)}\mathbf{1}])=(-1)^{n_1+n_2+\cdots+n_m}x^mx^{m-1}\cdots x^1,
\end{align*}
where $x^1,x^2,...,x^m\in \mathfrak{g}^0$.
Moreover, $g$-twisted Zhu's algebra of a simple affine vertex algebra $L_{k}(\mathfrak{g})$ is $U(\mathfrak{g}^0)/\langle [U(\mathfrak{g})v_{{\text{sing}}}]\rangle$, where $[v_{{\text{sing}}}]$ means the equivalence class of singular vector $v_{sing}$ generating the maximal submodule of $V^k(\mathfrak{g})$ in $A_g(L_{k}(\mathfrak{g})).$

Now we compute the $\hat{\sigma}$-twisted Zhu's algebra $A_{\hat{\sigma}}(L_{-\frac43}(\s_2))$. We denote by $L_{\mathfrak{sl}_{2}}(\ell,0) $ or $L(\ell,0)$ the vertex algebra $L_{\ell}(\mathfrak{sl}_{2})$ at level $\ell$, and denote  the admissible irreducible $L_{\ell}(\mathfrak{sl}_{2})$-module by $L_{\mathfrak{sl}_{2}}(\ell,j)$ or $L(\ell,j)$, where $j$ is the Dynkin label of finite part of the admissible weight of $\widehat{\mathfrak{sl}_{2}}$. Note that  $A_{\hat{\sigma}}(L_{-\frac43}(\s_2))$ coincides with $A_{\hat{\sigma},\hat{\sigma}}(L_{\mathfrak{sl}_2}(-\f43,0))$.
By above general argument, we have
\begin{align*}
    A_{\hat{\sigma}}(L_{-\frac43}(\s_2))\cong U(h)/\langle[ U(\mathfrak{g})v_{{\text{sing}}}] \rangle.
\end{align*}
The weight zero singular vector is \cite{creutzig2013modular}
\begin{align}\label{sing}
    9h_{(-1)}^3+18h_{(-2)}h_{(-1)}-16h_{(-3)}-36f_{(-1)}h_{(-1)}e_{(-1)}-24e_{(-2)}f_{(-1)}+96f_{(-2)}e_{(-1)}\mathbf{1}.
\end{align}
Using (\ref{fu1}), the equivalence class of (\ref{sing}) in $A_{\hat{\sigma}}(L_{-\frac43}(\s_2))$    can be written in terms of $h$. To that end, for $[f_{(-1)}h_{(-1)}e_{(-1)}]$ we have:
\begin{align*}
  [f_{(-1)}h_{(-1)}e_{(-1)}]&=[(-\frac12 f_{(0)}+\frac18 f_{(1)}-\frac{1}{16}f_{(2)}+\cdots)h_{(-1)}e_{(-1)}]  \\
                            &=[\frac12h_{(-1)}^2-\frac12h_{(-1)}+\frac16-\frac{5}{12}h_{(-1)}+\frac16]         \\
                            &=[\frac12h_{(-1)}^2-\frac{11}{12}h_{(-1)}+\frac13];
\end{align*}
for $[e_{(-2)}f_{(-1)}]$ we have:
\begin{align*}
    [e_{(-2)}f_{(-1)}]&=[(-\frac12e_{(-1)}+\frac18e_{(0)}-\frac{1}{16}e_{(1)}+\cdots)f_{(-1)}]\\
                      &=[\frac38h_{(-1)}+\frac16];
\end{align*}
for $[f_{(-2)}e_{(-1)}]$ we have:
\begin{align*}
    [f_{(-2)}e_{(-1)}]=[-\frac38h_{(-1)}+\frac16].
\end{align*}
Combining everything together, we get (\ref{sing}) equals
\begin{align}\label{sing21}
    [9h_{(-1)}^3+18h_{(-2)}h_{(-1)}-16h_{(-3)}+18h_{(-1)}^2+12h_{(-1)}].
\end{align}
The image of $(\ref{sing21})$ under $F$ is
\begin{align*}
    h(h+\frac23)(h-\frac{2}{3}).
\end{align*}
Thus, $A_{\hat{\sigma}}(L_{-\frac43}(\s_2)) \cong \mathbb{C}[h]/\langle h(h+\frac23)(h-\frac{2}{3}) \rangle$.

Consider the $\hat{\sigma}$-invariant subspace of $M^1$ and $M^2$. One has usual Zhu's bimodule $A((M^i)^{\hat{\sigma}})
=(M^i)^{\hat{\sigma}}/O((M^i)^{\hat{\sigma}})$, where $i=2,3$. From above definition, we have
\begin{align*}
&A_{\hat{\sigma},\hat{\sigma}}(M^2)=M^2/O\left((M^2)^{\hat{\sigma}}\right),\\
&A_{\hat{\sigma},\hat{\sigma}}(M^3)=M^3/O\left((M^3)^{\hat{\sigma}}\right).
\end{align*}

\subsection{Relation between Dong-Li-Mason's Zhu's bimodules and twisted Zhu's bimodules}
Let $\omega$ be the original Sugawara Virasoro vector of $L(\ell,0)$. Set $\omega_z=\omega+\f12z h(-2)\bd{1}\in L(\ell,0),$ where $z$ is a complex number. Then $\omega_z$ is a Virasoro vector of $L(\ell,0)$ with the central charge $c_{\ell,z}=c_{\ell}-6\ell z^2$. Let $z$ be a positive rational number less than $1$. Note that the vertex operator algebra $(L(\ell,0),Y,\bd{1},\omega_z)$ is $\Q$-graded instead of $\Z$-graded. In \cite{dong1997vertex} authors extend the definition of Zhu's $A(V)$-theory of one-to-one correspondence between the set of equivalence classes of irreducible admissible $V$-modules and the set of equivalence classes of irreducible $A(V)$-modules and Zhu-Frenkel's $A(M)$-theory for fusion rules to any $\Q$-graded vertex operator algebra $V$.
Denote the Zhu's algebra and bimodule of $V$ by $A^{\dlm}(V)=V/V\circ_{\dlm}V$ and $A^{\dlm}(M)=M/V\circ_{\dlm}M.$

Assume
\[\exp \p{\sum_{n\geq 1}\f{\f14h_{(n)}}{-n}(-z)^{-n}}=\sum_{n\geq 0}u_{(n)}z^{-n}.\]
For brevity, we suppress the subscript of $\Delta_{-\f12}$ and denote it by $\Delta$. Let $\epsilon(e)=\epsilon(f)=0$ and $\epsilon(h)=1$. We have
\begin{align*}
  \Delta(1)(a\circ_{\te{dlm}} m) &=\Delta(1)\res{z=0}{\p{Y(a,z)\f{(1+z)^{[\wt (a)_{\te{}}]}}{z^{1+\epsilon(a)}}m}} \\
  &=\res{z=0}{\p{\Delta(1)Y(a,z)\f{(1+z)^{[\wt(a)_{\te{}}]}}{z^{1+\epsilon(a)}}m}}\\
  &=\res{z=0}{Y(\Delta(z+1)a,z)\f{(1+z)^{[\wt(a)_{\te{}}]}}{z^{1+\epsilon(a)}}\Delta(1)m}\\
  &=\sum_{n\geq 0}\res{z=0}{Y(u_{(n)}a,z)\f{(1+z)^{[\wt(a)_{\te{}}]+\lambda-n}}{z^{1+\epsilon(a)}}\Delta(1)m}\\
  &=\sum_{n\geq 0}\res{z=0}{Y(u_{(n)}a,z)\f{(1+z)^{[\wt(u_{(n)}a)]+\lambda}}{z^{1+\epsilon(a)}}\Delta(1)m}\\
  &=(\Delta(1)a)\circ_{\hat{\sigma},\hat{\sigma}}(\Delta(1)m).
  \end{align*}
where $\f{h_{(0)}}{4}a=\lambda a.$
Using  similar arguments, we get
\begin{align*}
  \Delta(1)(a*_{\te{dlm}}m) &=(\Delta(1)a)*_{\hat{\sigma},\hat{\sigma}}(\Delta(1)m), \\
  \Delta(1)(m*_{\te{dlm}}a)
  &=(\Delta(1)m)*_{\hat{\sigma},\hat{\sigma}}(\Delta(1)a),
\end{align*}
where $*_{\te{dlm}}$ is the bimodule action defined by \cite{dong1997vertex}.
Thus we have the following result.
\begin{prop}\label{isozhu}
The map $V\rightarrow V, a\mapsto \Delta(1)a,$ induces an algebra isomorphism
\[A^{\te{dlm}}(V)\iso A_{\hat{\sigma}}(V).\]
The map $M^1\rightarrow M^1, m\mapsto \Delta(1)m$ induces an $A_{\hat{\sigma},\hat{\sigma}}(V)(\cong A^{\te{dlm}}(V))$-bimodule isomorphism
\[A^{\te{dlm}}(M^1)\iso A_{\hat{\sigma},
\hat{\sigma}}(M^1)\]
\end{prop}

\begin{ex}
  Note $\Delta(1)h(-1)v=(h-\f23)v$, and $\Delta(1)f(0)v=f(0)v$.
If $M^1=L_{\mathfrak{sl}_2}(-\f43,-\f23)$, we have
\[A_{\hat{\sigma},\hat{\sigma}}(M^1)\cong \C[x,y]/\la y, (x-\f23)x\ra. \]
If $M^1=L_{\s_2}(-\f43,-\f43)$, we have
\[A_{\hat{\sigma},\hat{\sigma}}(M^1)\cong \C[x,y]/\la y, (x-\f23)\ra.\]
Here, we identify $f(0)$ and $h(-1)$ with $y$ and $x$.
\end{ex}

\subsection{Fusion rules among twisted modules}



According to \cite{dong1997vertex},
\[A^{\text{dlm}}(L(\ell,0))\cong \C[x]/\la \prod_{r=0}^{p-2}\prod_{s=0}^{q-1} (x-r+st)\ra.\] Since $\Delta(1)h(-1)v=(h(-1)+\f12\ell)v$, by Proposition \ref{isozhu}, we have
\begin{align}\label{formula59}
	A_{\hat{\sigma}}(L(\ell,0))\cong \C[x]/\la\prod_{r=0}^{p-2}\prod_{s=0}^{q-1}(x+\f12\ell-r+st)\ra.
\end{align}
Similarly, since $\Delta(1)f(0)v=f(0)v$,
$A_{\hat{\sigma},\hat{\sigma}}(L_{\s_2}(\ell,j))$ is isomorphic to the quotient space of $\C[x,y]$ modulo the subspace
\[\C[x,y]y^n+\C[x]g_{j,0}(x,y)+\C[x]g_{j,1}(x,y)+\cdots +\C[x]g_{j,n-1}(x,y)\]
where $g_{j,i}=y^i\prod_{r=0}^{p-n-1}\prod_{s=0}^{q-k}(x+\f12\ell-r-i+st).$
The left and right actions of $A_{\hat{\sigma}}(L(\ell,0))$ is given by (\ref{biact}).

\begin{thm}
  For admissible weights $j_i=n_i-(k_i-1)t$ $(i=1,2)$, the fusion rules are given as follows:
  \begin{align}\label{twfus}L(\ell,j_1)\times \sigma^{-\f12}(L(\ell,j_2))=\sum_{i=\max\{0,n_1+n_2-p\}}^{\min\{n_1-1,n_2-1\}}\sigma^{-\f12}(L(\ell,j_1+j_2-2i))\end{align}
 if $0\leq k_2-1\leq q-k_1$, and $L(\ell,j_1)\times \sigma^{-\f12}(L(\ell,j_2))=0$ otherwise.
\end{thm}

\begin{proof}
  For any admissible weight $j$, let $\C v_j'$ be the one-dimensional module for Lie algebra $\C h$ such that $h v_j' =(j-\f12\ell)v_j'.$ We then calculate the $A(L(\ell,0))$-module $A_{\hat{\sigma},\hat{\sigma}}(L(\ell,j_1))\otimes_{A_{\hat{\sigma}}(L(\ell,0))}\C v_{j_2}'.$ Using the result above this Theorem, we have
  \[A_{\hat{\sigma},\hat{\sigma}}(L(\ell,j_1))\otimes_{A_{\hat{\sigma}}(L(\ell,0))} \C v_{j_2}'\cong \C[x,y]/J\]
  where $J$ is the subspace of $\C[x,y]$ spanned by
  \[\{x-j_2+\f12\ell,\C[x,y]y^{n_1}, g_{j_1,i}(j_2,1)\C[x]y^i,i=0,1,...,n_1-1\}.\]
  Then the result follows from the similar argument of \cite[Theorem 4.7]{dong1997vertex}.
\end{proof}
\begin{ex}

let $j_i=-k_i\f23$ $(i=1,2)$ where $k_i=0,1,2$,
 we have
\begin{align}\label{twfusl2}
  L_{\s_2}(-\f43,j_1)\times \sigma^{-\f12}(L_{\s_2}(-\f43,j_2))=\sigma^{-\f12}(L_{\s_2}(-\f43,j_1+j_2))
\end{align}
for $0\leq k_2\leq 2-k_1.$	
\end{ex}

\subsection{Verlinde formula for $L_{-\frac{4}{3}}(\mathfrak{sl}_{2})$ }

Let $\{\lambda_1:=-\frac{4}{3}\Lambda_0,
\lambda_2:=-\frac{2}{3}\Lambda_0-\frac23\Lambda_1,
\lambda_3:=-\frac{4}{3}\Lambda_1\}$.
The $S$-transformation is given by
 \[\chi_{\lambda_i}(-\frac{1}{\tau})=
 \sum_{j=1}^{3}a_{ij}\sigma^{-\f12}(\chi_{\lambda_j}(\tau))\]
\[
  (a_{ij}) =\frac{\sqrt{3}}{3}
  \left( {\begin{array}{ccc}
    -1 & 1 & -1 \\
    1 & -e^{-\frac{2}{3}\pi i}& e^{\frac{2}{3}\pi i} \\
    -1& e^{\frac{2}{3}\pi i}& -e^{-\frac{2}{3}\pi i}\\
  \end{array} } \right)
.\]
The fusion rules are
\[N_{0i}^j=
\left(
  \begin{array}{ccc}
    1 & 0 & 0 \\
    0 & 1 & 0 \\
    0 & 0 & 1 \\
  \end{array}
\right),\;\;N_{1i}^j=\left(
    \begin{array}{ccc}
      0 & 1 & 0 \\
      0 & 0 & 1 \\
      0 & 0 & 0 \\
    \end{array}
  \right),\;\;
N_{2i}^j=\left(
             \begin{array}{ccc}
               0 & 0 & 1 \\
               0 & 0 & 0 \\
               0 & 0 & 0 \\
             \end{array}
           \right),
\]
while the Verlinde formula and $S$-matrix would give us the following fusion rules:
\[N_{0i}^j=
\left(
  \begin{array}{ccc}
    1 & 0 & 0 \\
    0 & 1 & 0 \\
    0 & 0 & 1 \\
  \end{array}
\right),\;\;N_{1i}^j=\left(
    \begin{array}{ccc}
      0 & 1 & 0 \\
      0 & 0 & 1 \\
      -1 & 0 & 0 \\
    \end{array}
  \right),\;\;
N_{2i}^j=\left(
             \begin{array}{ccc}
               0 & 0 & 1 \\
               -1 & 0 & 0 \\
               0 & -1 & 0 \\
             \end{array}
           \right).
\]

The cause of the negative coefficients in the above fusion rules is explained in \cite{creutzig2013modular}. In order to obtain the positive coefficients, one need include more objects in our category, i.e., {\em atyptical} and {\em typical} admissible modules and take the convergent regions of the characters of these modules into consideration.

%

\section{Twisted Zhu's bimodules of contragredient modules of highest weight modules }\label{section4}
\subsection{Motivation}
Given an admissible irreducible highest weight $L_{\mathfrak{sl}_{2}}(-\frac{4}{3},0)$-module $L_{\mathfrak{sl}_{2}}(-\frac{4}{3},-\frac{4}{3})$, one can obtain the ordinary twisted modules and contragredient modules by taking integer and half-integer spectral flow respectively (Figure \ref{spectraldiagram}). In general, let $\{M_{i}\}$ be the collection of all highest weight modules for $L_{\mathfrak{sl}_{2}}(\ell,0)$ at the admissible level, one can obtain all the $\mathbb{Z}_{2}$-twisted modules of $L_{\mathfrak{sl}_{2}}(\ell,0)$ at the admissible level either from $\{\sigma^{-\frac{1}{2}}(M_{i})\}$ or $\{\sigma^{\frac{1}{2}}(M_{i}^*)\}$.
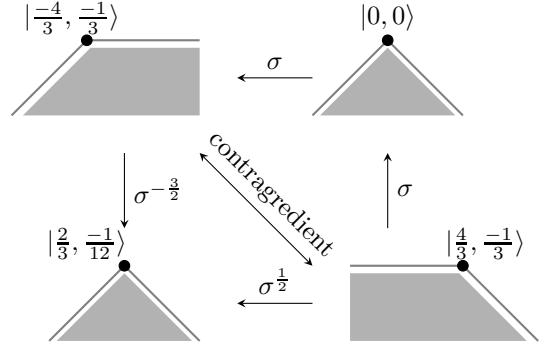
\begin{figure}[h]
\centering
\begin{tikzpicture}
\draw[gray, thick] (0,0) -- (-1,-1);
\draw[gray, thick] (0,0) -- (1.5,0);
\draw[gray,thick]  (4,0) --(3,-1);
\draw[gray,thick]  (4,0) --(5,-1);
\draw[gray,thick](0.5,-3)--(-0.5,-4);
\draw[gray,thick]  (0.5,-3) --(1.5,-4);
\draw[gray,thick]  (5,-3) --
			(3.5,-3);
			\draw[gray,thick]  (5,-3) --
			(6,-4);
			\filldraw [black] (0,0) circle (2pt);
			\filldraw [black] (4,0) circle (2pt);
			\filldraw [black] (0.5,-3) circle (2pt);
			\filldraw [black] (5,-3) circle (2pt);
			\node[] (v1) at (-0.2,0.3) {$|\frac{-4}{3},\frac{-1}{3}\rangle$};
			\node[] (v2) at (4,0.3) {$|0,0\rangle$};
			\node[] (v3) at (0,-2.7) {$|\frac{2}{3},\frac{-1}{12}\rangle$};
			\node[] (v4) at (5.3,-2.7) {$|\frac{4}{3},\frac{-1}{3}\rangle$};
			\draw [stealth-] (2,-0.5) -- node[above]{$\sigma$}(3,-0.5);
			\draw [-stealth] (0.5,-1.5) -- node[right]{$\sigma^{-\frac{3}{2}}$}(0.5,-2.5);
			\draw [stealth-] (2,-3.5) -- node[above]{$\sigma^{\frac{1}{2}}$}(3,-3.5);
			\draw [stealth-] (4,-1.5) -- node[right]{$\sigma$}(4,-2.5);
			\draw [stealth-stealth] (1.5,-1.5) -- node[above ,sloped]{$\mathrm{contragredient}$}(3,-3);
			\fill[fill=gray!60] (0.05,-0.1) -- (-0.85,-1) -- (1.5,-1) -- (1.5,-0.1) -- cycle;
			\fill[fill=gray!60] (4.95,-3.1) -- (3.5,-3.1) -- (3.5,-4) -- (5.85,-4) -- cycle;
			\fill[fill=gray!60] (4,-0.1) -- (3.1,-1) -- (4.9,-1) -- cycle;
			\fill[fill=gray!60] (0.5,-3.1) -- (-0.4,-4) -- (1.4,-4) -- cycle;
\end{tikzpicture}
\caption{The relation among admissible irreducible highest weight $L(-4/3,0)$-modules, contragredient modules and ordinary $\mathbb{Z}_{2}$-twisted modules, where each state labelled by $|\lambda,\Delta\rangle$, its $\mathfrak{sl}_{2}$-weight $\lambda$ and conformal dimension $\Delta$. }\label{spectraldiagram}
\end{figure}
In this Section, we will calculate the twisted Zhu's bimodules of the contragredient modules of highest weight modules and the fusion rules among them. The idea is to use the similar isomorphism in Proposition \ref{isozhu}. To that end, we calculate the untwisted Zhu's bimodule first and fusion rules among untwisted modules.
\subsection{Untwisted Zhu's bimodules}
Let $M$ be a $V$-module. Let $O(M)$ be the linear span of elements of type
\[\res{z}{\p{Y(a,z)\f{(z+1)^{\wt a}}{z^2}v}}.\]
In particular, for $V=L_{\s_2}(-\f43,0)$. The $O(M)$ is spanned by
\begin{align*}
  \res{z}{\p{Y(e,z)\f{(z+1)}{z^{m+1}}}}v &=(e(-m-1)+e(-m))v, \\
  \res{z}{\p{Y(f,z)\f{(z+1)}{z^{m+1}}}}v &=(f(-m-1)+f(-m))v, \\
  \res{z}{\p{Y(h,z)\f{(z+1)}{z^{m+1}}}}v &=(h(-m-1)+h(-m))v, \\
\end{align*}
for any positive integer $m$ and for $v\in M$.

By \cite{FrenkelZhu}, one has the following isomorphism $F$
\begin{align*}
  F:L(j)\otimes U( \mathfrak{g})&\rightarrow A(V_{\mathfrak{g}}(\ell,j)), \\
  F:v\otimes a_1\cdots a_n&\mapsto [a_n(-1)\cdots a_1(-1)v].
\end{align*}
\noindent whose inverse is given by
\begin{align*}
F^{-1}:	A(V_{\mathfrak{g}}(\ell,j))&\rightarrow L(j)\otimes U( \mathfrak{g}),\\
F^{-1}:[a_1(-1-i_1)\cdots a_n(-1-i_n)v]&\rightarrow (-1)^{i_1+\cdots +i_n}v\otimes a_n\cdots a_1,
 \end{align*}
where the tensor products are over $U(\mathfrak{g})$.

 Consider the example of $L_{\mathfrak{sl}_2}(-\f43,0)$. It has three irreducible modules in category $\mathcal{O}$, i.e., $L_{\mathfrak{sl}_2}(-\f43,0), L_{\mathfrak{sl}_2}(-\f43,-\f23)$, $L_{\mathfrak{sl}_2}(-\f43,-\f43)$.
By \cite{malikov1986singular} the maximal submodule of $V_{\mathfrak{sl}_2}(-\f43,-\f23)$ is generated by singular vectors $v_{1}$ and $v_{2}$,
\begin{align}
&v_{1}=\f29e_{(-2)}-\f13e_{(-1)}h_{(-1)}+e_{(-1)}e_{(-1)}f_{(0)}\label{sing2}\\
&v_{2}=-\f{10}{9}f_{(-1)}-\f53h_{(-1)}f_{(0)}+e_{(-1)}f_{(0)}^2,\label{sing4}	
\end{align}
and the maximal submodule of $V_{\mathfrak{sl}_2}(-\f43,-\f43)$ is generated by singular vectors $e_{(-1)}$ and
\begin{align}
&\f{280}{81}f_{(-2)}+\f{70}{27}h_{(-2)}f_{(0)}-\f{10}{9}e_{(-2)}f_{(0)}^2\nonumber\\
&+\f{140}{27}h_{(-1)}f_{(-1)}+\f{35}{9}h_{(-1)}^2f_{(0)}-\f53h_{(-1)}e_{(-1)}f_{(0)}^2\label{sing3}\\
&-\f{70}{9}e_{(-1)}f_{(-1)}f_{(0)}-\f{10}{3}e_{(-1)}h_{(-1)}f_{(0)}^2+e_{(-1)}^2f_{(0)}^3	\nonumber
\end{align}

We next use the singular vectors to compute $A(L_{\s_2}(-\f43,-\f23)), A(L_{\s_2}(-\f43,-\f43)).$
For $A(L_{\s_2}(-\f43,-\f23))$, the preimage of equivalence classes of $v_1$ and $v_2$ are
\begin{align}
& -\f29 v\otimes e-\f13 v\otimes he+fv\otimes e^2, \label{rela1}\\
& -\f{10}{9}v\otimes f-\f53 fv\otimes h+f^2v\otimes e.\label{rela2}
\end{align}
Thus
\[A(L_{\s_2}(-\f43,-\f23))=(L(-\f23)\otimes U(\s_2))/I_1, \]
where $I_1$ is generated by (\ref{rela1}) and (\ref{rela2}), and the tensor products are over $U(\mathfrak{g})$.
For $A(L_{\s_2}(-\f43,-\f43))$, the preimage of equivalence classes of (\ref{sing3}) is
\begin{align*}
  -\f{280}{81}v\otimes f-\f{70}{27}fv\otimes h+\f{10}{9}f^2v\otimes e&+\f{140}{27}v\otimes hf+\f{35}{9}fv\otimes h^2\\ &-\f53f^2v\otimes eh-\f{70}{9}fv\otimes fe-\f{10}{3}f^2v\otimes he+f^3v\otimes e^2.
\end{align*}

Denote by $L(-\f23)^*$ and $L(-\f43)^*$ the dual of highest weight $\s_2$-modules with weights $-\f23$ and $-\f43$.
We now consider the $A(L_{\s_2}(-\f43,0))$-module $L(-\f23)^*\otimes A(L_{\s_2}(-\f43,-\f23)) $.
Let $v'$ be the lowest weight vector of $L(-\f23)^*$. Then $I_1\otimes v'\cong \la -\f{10}{9}v\otimes ev'+fv\otimes e^2v', -\f{10}{9}fv\otimes v'+f^2v\otimes ev'\ra.$
It is isomorphic to \[A(L_{\s_2}(-\f43,-\f23))\otimes L(-\f23)^*\cong \mathbb{C} (v\otimes v').\]
Thus we have
\[L_{\s_2}(-\f43,-\f23)\times \p{L_{\s_2}(-\f43,-\f23)}^*=L_{\s_2}(-\f43,0). \]
Similarly, we also have
\begin{align}
& L_{\s_2}(-\f43,-\f43)\times \p{L_{\s_2}(-\f43,-\f43)}^*=L_{\s_2}(-\f43,0),\\
& L_{\s_2}(-\f43,-\f23)\times \p{L_{\s_2}(-\f43,-\f43)}^*=\p{L_{\s_2}(-\f43,-\f23)}^*,\\
& L_{\s_2}(-\f43,-\f43)\times \p{L_{\s_2}(-\f43,-\f23)}^*=L_{\s_2}(-\f43,-\f23),
\end{align}

\subsection{The contragredient modules of the highest weight modules}
As one can see from above example, directly computing Zhu's bimodules depends on the explicit form of singular vectors. In practice, it is extremely tedious to convert the singular vectors given in \cite{malikov1986singular} into their normal forms. It was noted in \cite{dong1997vertex} that the fusion rules among the admissible modules remain the same after a shift of the conformal vector.  By a proper shift of the conformal vector, there are nice and compact projection formulas for singular vectors (\cite{fuchs1989two}, \cite{dong1997vertex}) which can help us to compute fusion rules avoiding finding explicit form of singular vectors.

Since
\begin{align}
   & [f(0),e(0)^{\gamma}]=-(\gamma e(0)^{\gamma-1}+\gamma(\gamma-1)e(0)^{\gamma-1})\label{rel1},\\
   & [h(0),e(0)^{\gamma}]=2\gamma e(0)^{\gamma}\label{rel2},
\end{align}
then by using the similar argument as in \cite{malikov1986singular} we have
\begin{prop}
  Let $j=n-1-(k-1)$ where $n$ and $k$ are positive integers satisfying $1\leq n\leq p-1$, $1\leq k\leq q$ and let $v$ be a highest weight vector of the Verma module $(M(\ell,j))^*$. Set
  \begin{align*}
     & E_1(n,k)=e(0)^{n+(k-1)t}f(-1)^{n+(k-2)t}e(0)^{n+(k-3)t}f(-1)^{n+(k-4)t}
     \cdots f(-1)^{n-(k-2)t}e(0)^{n-(k-1)t},\\
     &
     E_2(n,k)=f(-1)^{p-n+(q-k)t}e(0)^{p-n+(q-k-1)t}f(-1)^{p-n+(q-k-2)t}e(0)^{p-n+(q-k-3)t}\\
     &\quad \cdots e(0)^{p-n-(q-k+1)t}f(-1)^{p-n-(q-k)t}.
  \end{align*}

\noindent Then $v_{-j,1}=E_1(n,k)v,v_{-j,2}=E_2(n,k)v$ are singular vectors of $(M(\ell,j))^*$.

\end{prop}
\noindent Basically, we just interchange $e$ and $f$ based on the corresponding result in the case of highest weight modules.

Next, we consider the projection formula. Let $P_1$ be the projection $\hat{\mathfrak{g}}$ onto $\mathfrak{g}$ such that $P_1(a\otimes t^n)=a$ for any $a\in \mathfrak{g}$ and $P_1(c)=0$. Let $H_\al=fe-\al h-\al(\al+1)$.
\begin{prop}\cite{malikov1986singular}
  The following projection formulas hold:
  \begin{align*}
     & P_1(E_1(n,k))=\p{\prod_{r=1}^{n}\prod_{s=1}^{k-1}H_{-r-st}}e^n \\
     &
     P_1(E_2(n,k))=\p{\prod_{r=0}^{p-n-1}\prod_{s=1}^{q-k}H_{r+st}}f^{p-n}.
  \end{align*}
\end{prop}

Define subalgebra
\[T_-=\C e+t^{-1}\C[x^{-1}]\otimes \mathfrak{g}.\]
 Let $B_0=\C(t^{-1}+1)\otimes e+(x^{-2}+x^{-1})\C[x^{-1}]\otimes \mathfrak{g}.$
Since $B_0$ is an ideal of $N_-$, $U(N_-)B_0=B_0U(N_-)$ is an ideal of $U(N_-)$. Set $L_0=N_-/B_0$. Define
\[T_+=e(0)+B_0, T_-=f(-1)+B_0, T_0=h(-1)+B_0.\] They obey the following $\s_2$-relationships
\[[T_0,T_+]=-2T_+,\;\; [T_0,T_-]=2T_-,\;\;[T_+,T_-]=T_0.\]

\noindent   Define $G_\al=T_-T_+-\al T_0+\al(\al+1)$.
They satisfy the following relationships
\begin{align*}
   & G_\al G_\be=G_\be G_\al,\ju T_+^mG_\al=G_{\al-m}T_+^m,\ju T_-^m G_\al=G_{\al+m}T_-^m, \\
   & T_-^m T_+^m=G_0G_1\cdots G_{m-1},\ju T_+^mT_-^m=G_{-1}G_{-2}\cdots G_{-m},
\end{align*}
for any complex numbers $\al, \be$ and for any positive integer $m$.

Let $P$ be the natural quotient map from $U(N_-)$ onto $U(L_0)$.
Using the similar method as suggested in \cite{malikov1986singular} we obtain
\begin{prop}
  The following formulas hold:
  \begin{align*}
     & P(E_1(n,k))=\p{\prod_{r=1}^{n}\prod_{s=1}^{k-1}G_{-r-st}}T_+^n \\
     & P(E_2(n,k))=\p{\prod_{r=0}^{p-n-1}\prod_{s=1}^{q-k}G_{r+st}}T_-^{p-n}.
  \end{align*}
\end{prop}

\subsection{Fusion rules}
For the contragredient modules of the highest weight modules, we choose the new conformal vector $\omega_z=\omega-\f12zh(-2)\bf{1}$, where $0<z<1$.

Let $M$ be any weak $V(\ell,\C)$-module. Since
\[\wt h(-1)=1,\;\; \wt e(-1)=1+z,\;\;\wt f(-1)=1-z, \]
we have
\begin{align*}
  \res{z}{\f{(1+z)^{[\wt f]}}{z^m}Y(f,z)u}&=f(-m)u \\
   \res{z}{\f{(1+z)^{[\wt e]}}{z^m}Y(e,z)u} &=(e(-m)+e(1-m))u,\\
   \res{z}{\f{(1+z)^{\wt h}}{z^{m+1}}Y(h,z)u}&= (h(-m-1)+h(-m))u
\end{align*}

\noindent for any positive integer $m$ and for $u\in M$.

\begin{prop}\label{bimod}
  Let $j=n-1-(k-1)t$ be an admissible weight. Then the $A(L(\ell,0))$-bimodule $A((L(\ell,j))^*)$ is isomorphic to the quotient space of $\C[x,z]$ modulo the subspace
  \[\C[x,z]z^n+\C[x]f'_{j,0}(x,z)+\C[x]f'_{j,1}(x,z)+\cdots+\C[x]f'_{j,n-1}(x,z)\]
where $f'_{j,i}(x,z)=z^{i}\prod_{r=0}^{p-n-1}\prod_{s=0}^{q-k}(x+r+i-st).$ The left and right actions of $A(L(\ell,0))$ on $A(L(\ell,j)^*)$ are given by (\ref{biact}).
\end{prop}
\begin{proof}
  First, $(M(\ell,j))^*\cong U(N_-)$ as a vector space.

  We have
  \[O((M(\ell,j))^*)\cong f(-1)U(N_-)+B_0U(N_-).\]

Recall $v_{-j,1},v_{-j,2}$ are two singular vectors of $(M(\ell,j))^*$. Then we have
\begin{align*}
   & A((L(\ell,j))^*)\cong (M(\ell,j))^*/(O((M(\ell,j))^*+U(N_-)v_{-j,1}+U(N_-)v_{-j,2})\\
    & \cong U(N_-)/(B_0U(N_-)+f(-1)U(N_-)+U(N_-)E_1(n,k)+U(N_-)E_2(n,k))
\end{align*}
as $A((L(\ell,0))^*)$-bimodules. Note that $U(N_-)/B_0U(N_-)\cong U(L_0)$. Thus
\[A((L(\ell,j))^*)\cong U(L_0)/(U(L_0)P(E_1(n,k))+U(L_0)P(E_2(n,k))+T_-U(L_0)).\]

For any nonnegative integers $a,b,d,$ using above relationships, we have (See Appendix B for the detail)
\begin{align}
\label{appen1}
\begin{split}
&\ju T_-^a T_0^b T_+^d P(E_1(n,k)) \\
& =
   T_-^a\p{\prod_{r=1}^{n}\prod_{s=1}^{k-1}(r+st+d)(T_0+r+st+d-1)}T_0^bT_+^{n+d} \ju \mod T_-U(L_0).	
\end{split}
\end{align}


\noindent Noticing that $r+st+d\neq 0$ for any $1\leq r\leq n$, $1\leq s\leq k-1$, $d\in \Z_{+}$ we obtain
\begin{align*}
\begin{split}
& \ju U(L_0)P(E_1(n,k))+T_-U(L_0) \\
   & =T_-U(L_0)+\sum_{d=0}^{\infty}\C[T_0]\p{\prod_{r=0}^{n-1}\prod_{s=1}^{k-1}(T_0+r+st+d)}T_+^{n+d}.	
\end{split}
\end{align*}

Similarly, let $a,b,d$ be any nonnegative integers. If $d<p-n$, we have (see Appendix B for details)
\begin{align}\label{appen2}
\begin{split}
&\ju T_-^a T_0^b T_+^d P(E_2(n,k))\\
   &=T_-^{a+p-n-d}(T_0+2(p-n-d))^b \prod_{r=0}^{p-n-1}\prod_{s=1}^{q-k}\prod_{i=1}^{d} G_{r+st-p+n}G_{-i-p+n+d}\ju \mod T_-U(L_0).
\end{split}
\end{align}


\noindent If $d=m+p-n$ for some $m\in\Z_+$, we have
\begin{align*}
   & \ju T_-^a T_0^b T_+^d P(E_2(n,k)) \\
   &=T_-^a\prod_{r=1}^{p-n}\prod_{s=0}^{q-k}(st-m-r)(-T_0+st-m-r+1)T_0^b T_+^m  \ju \mod T_-U(L_0).\\
\end{align*}


Since $-st+m+r-1\neq 0$ for any $1\leq r\leq p-n$, $0\leq s\leq q-k$, we obtain
\begin{align*}
   &\ju U(L_0)P(E_2(n,k))+T_+U(L_0)  \\
   & =T_-U(L_0)+\sum_{m=0}^{\infty}\C[T_0]\p{\prod_{r=0}^{p-n-1}\prod_{s=0}^{q-k}(T_0-st+m+r)}T_+^m.
\end{align*}

Thus
\begin{align*}
   & \ju U(L_0)P(E_1(n,k))+U(L_0)P(E_2(n,k))+T_-U(L_0) \\
   & \subset T_+U(L_0)+U(L_0)T_+^n+\sum_{i=0}^{n-1}\C[T_0]\p{\prod_{r=0}^{p-n-1}\prod_{s=0}^{q-k}(T_0-st+m+r)}T_+^m.
\end{align*}

On the other hand, since $-r-st-d\neq s't'-m-r'$ for any $0\leq r\leq n-1, 1\leq s\leq k-1$, $0\leq r\leq p-n-1, 0\leq s\leq q-k$, $d,m\in\Z_+$, $\prod_{r=0}^{p-n-1}\prod_{s=0}^{q-k}(x-st+m+r)$ and $\prod_{r=0}^{p-n-1}\prod_{s=0}^{q-k}(x-st+m+r)$ are relatively prime. Then we obtain
\[\C[T_0]T_+^{n+i}\subset U(L_0)P(E_1(n,k))+U(L_0)P(E_2(n,k))+T_-U(L_0)\]
for any $i\in \Z_+$. This shows that
\begin{align*}
   & U(L_0)P(E_1(n,k))+U(L_0)P(E_2(n,k))+T_-U(L_0) \\
   & \supset T_-U(L_0)+U(L_0)T_+^n+\sum_{i=0}^{n-1}\C[T_0]\p{\prod_{r=0}^{p-n-1}\prod_{s=0}^{q-k}(T_0-st+m+r)}T_+^i.
\end{align*}

\noindent Set $x=T_0$, $y=T_-$. Then the Proposition follows from the similar argument in Proposition 4.3, Lemma 4.5 (\cite{dong1997vertex}).
\end{proof}

We have the fusion rules:
\begin{thm}
For admissible weight $j_i=n_i-1-(k_i-1)t$ $(i=1,2)$, the fusion rules are given as follows:
  \begin{align}
   & L(\ell,j_1)\times L(\ell,j_2)=\sum_{i=\max\{0,n_1+n_2-p\}}^{\min\{n_1-1,n_2-1\}}L(\ell,j_1+j_2-2i), \label{fusion3}\\
   & (L(\ell,j_1))^*\times L(\ell,j_2)
   =L(\ell,j_2)\times (L(\ell,j_1))^*=
  \left\{
    \begin{array}{ll}
      L(\ell,-j_1+j_2) , & \hbox{if $n_2-n_1\geq 0$;} \\
      (L(\ell,j_1-j_2))^*, & \hbox{if $n_2-n_1<0$.}
    \end{array}
  \right. \label{fusion2}\\
& (L(\ell,j_1))^*\times (L(\ell,j_2))^*
 =\sum_{i=\max\{0,n_1+n_2-p\}}^{\min\{n_1-1,n_2-1\}}(L(\ell,j_1+j_2-2i))^*.\label{fusion}
 \end{align}

\end{thm}

\begin{proof}
(\ref{fusion3}) was proved in \cite{dong1997vertex}. We use the similar method to prove (\ref{fusion2}) and (\ref{fusion}).

We prove the (\ref{fusion2}). For any admissible weight $-j$, let $\C v_{-j}$ be the one-dimensional module for Lie algebra $\C h$ such that $h v_{-j}=-j v_{-j}$. Then $\C  v_{-j}$ is the lowest weight space of $L(\ell,-j)$. By Frenkel-Zhu's Theorem, we need calculate the $A(L(\ell,0))$-module $A(L(\ell,-j_1))\otimes_{A(L(\ell,0))} \C v_{j_2 }.$ Note $e v_{j_2}=0$. We get
\[A(L(\ell,-j_1))\otimes_{A(L(\ell,0))}\C v_{j_2}\cong \C[x,z]/J\]
where $J$ is the subspace of $\C[x,z]$ spanned by
\[\{x-j_2,z\}.\]
\noindent Thus, $\C[x,z]/J\cong v_{-j_1}\otimes v'_{j_2}.$ And $x* (v_{-j_1}\otimes v'_{j_2})=j_2-j_1$, as required.

  For (\ref{fusion}), let $\C v_{-j}$ be the one-dimensional module for Lie algebra $\C h$ such that $hv_{-j}=-jv_{-j}$.
  Using Proposition \ref{bimod} we get
  \[A(L(\ell,j_1))\otimes_{A(L(\ell,0))}\C v_{-j_2}\cong \C[x,z]/J\]
  where $J$ is the subspace of $\C[x,z]$ spanned by
\[\{x+j_2,\C[x,z]z^{n_1}, f'_{j_1,i}(-j_2,1)\C[x]z^i,i=0,1,...,n_1-1\}.\]

If $j_2$ does not satisfy the relation $0\leq k_2-1\leq q-k_1$, then
\[f'_{j_1,i}(-j_2,1)=\prod_{r=0}^{p-n_1-1}\prod_{s=0}^{q-k_1}(-j_2+r+i-st)\neq 0\]
for $0\leq i\leq n_1-1.$ Thus $A(L(\ell,-j_1))\otimes_{A(L(\ell,0))}\C v_{j_2}=0$ so that all the corresponding fusion rules are zero.

Suppose $0\leq k_2-1\leq q-k_1. $ As before $\C[x]z^i=0$ in $\C[x,z]/J$ if $f'_{j_1,i}(j_2,1)\neq 0.$ Noticing $f'_{j_1,i}(j_2,1)=0$ if and only if $-j_2+r+i-st=0$ for some $0\leq r\leq p-n_1-1$ and $0\leq s\leq q-k_1$. This implies that $r+i=n_2-1$. Thus $n_1+n_2-p\leq i\leq n_2-1.$ Therefore
\[\max\{0,n_1+n_2-p\}\leq i\leq \min\{n_1-1,n_2-1\}. \]
If $n_1+n_2-p\leq i\leq n_2-1$, then $\C[x]z^i$ is not zero in $\C[x,z]/J$. Thus
\[\C[x,z]/J\cong\oplus_{\max\{0,n_1+n_2-p\}\leq i\leq \min\{n_1-1,n_2-1\}}\C y^i. \]
From (\ref{biact}) we get $x*z^i=(-j_2-j_1+2i)z^i$, as required.

\end{proof}

\subsection{Fusion rules among twisted modules}
We follow the same notations as in previous subsection. Let $\hat{\sigma}'=e^{-\f{\pi ih_{(0)}}{2}}$. By using the similar arguments as in Section 3, one has
\begin{thm}
\begin{itemize}
  \item We have the following isomorphism
  \[A(L(\ell,j)^*) \cong A_{\hat{\sigma}',\hat{\sigma}'}(L(\ell,j)^*),\]
  via $\Delta_{\f12}(1)$.
  \item We also have
  \begin{align}
     &  (L(\ell,j_1))^*\times \sigma^{\f12}((L(\ell,j_2))^*)
 =\sum_{i=\max\{0,n_1+n_2-p\}}^{\min\{n_1-1,n_2-1\}}\sigma^{\f12}((L(\ell,j_1+j_2-2i))^*).\\
      & (L(\ell,j_1))^*\times \sigma^{\f12}(L(\ell,j_2))
   =
  \left\{
    \begin{array}{ll}
      \sigma^{\f12}(L(\ell,-j_1+j_2)) , & \hbox{if $n_2-n_1\geq 0$;} \\
      \sigma^{\f12}((L(\ell,j_1-j_2))^*), & \hbox{if $n_2-n_1<0$.}
    \end{array}
  \right.\\
      &
   L(\ell,j_2)\times \sigma^{-\f12}((L(\ell,j_1))^*)=
  \left\{
    \begin{array}{ll}
      \sigma^{-\f12}(L(\ell,-j_1+j_2)) , & \hbox{if $n_2-n_1\geq 0$;} \\
      \sigma^{-\f12}((L(\ell,j_1-j_2))^*), & \hbox{if $n_2-n_1<0$.}
    \end{array}
  \right.
  \end{align}
\end{itemize}
\end{thm}

\section{Twisted modules from spectral flow and their MLDEs}\label{section5}

In \cite{Hao2022}, we have the following result
\begin{thm}
	If $V$ is a quasi-lisse vertex superalgebra and $g$ is an automorphism of $V$ of finite order, then the supercharacter of its simple $g$-twisted module satisfies the twisted modular linear differential equation.
\end{thm}
\noindent In this section and the next, we shall provide examples for this Theorem, i.e., twisted modules coming from spectral flow. We also discuss their applications in physics.

Firstly, let  us review some useful facts about modular forms and modular differential operators. See any standard reference, or \cite{Kaneko:2013uga} for further details. The ordinary Eisenstein series are modular forms for the full modular group $\Gamma$ of weight $2k$ with $k\geq 2$. We define our Eisenstein series,
following the notation of \cite{Beem:2017ooy},
\begin{equation}
\mathbbm{E}_{k}(\tau)=-\frac{B_{2k}}{2k!}+\frac{2}{(2k-1)!}\sum_{n=1}^{\infty}\frac{n^{2k-1}q^{n}}{1-q^{n}}
\end{equation}
where $B_{2k}$ is the $2k$’th Bernoulli number. The ring of modular forms for the full modular group $\Gamma$ is freely generated by $\mathbbm{E}_{4}(\tau)$ and $\mathbbm{E}_{6}(\tau)$, so we
have,
\begin{equation}
\bigoplus_{k=0}^{\infty}M_{k}(\Gamma,\mathbb{C})=\mathbb{C}[\mathbbm{E}_{4}(\tau),\mathbbm{E}_{6}(\tau)]
\end{equation}
We also make use of a class of twisted Eisenstein series that are modular forms for certain congruence subgroups of $\Gamma$,
\begin{equation}
\mathbbm{E}_{k}\begin{bmatrix}\varphi\\ \vartheta\end{bmatrix}(\tau)\equiv-\frac{B_{k}(\lambda)}{k!}+\frac{1}{(k-1)!}\sum_{r\geq0}{}' \frac{(r+\lambda)^{k-1}\vartheta^{-1}q^{r+\lambda}}{1-\vartheta^{-1}q^{r+\lambda}}+\frac{(-1)^{k}}{(k-1)!}\sum_{r\geq1}\frac{(r-\lambda)^{k-1}\vartheta q^{r-\lambda}}{1-\vartheta q^{r-\lambda}}
\end{equation}
where $\varphi=e^{2\pi i \lambda}$ with $\lambda\in [0,1)$ and now $B_{k}(x)$ is the $k$’th Bernoulli polynomial. The prime in the first summation means that the $r=0$ term should be omitted when $\varphi=\vartheta=1$. The spaces of modular forms for $\Gamma(2)$, $\Gamma^{0}(2)$ all admit a simple description
in terms of theta functions. For example,
\begin{equation}	M_{2k}(\Gamma^{0}(2))=\text{span}_{\mathbb{C}}\left\{\bar{\Theta}_{r,s}(\tau)|r+s=k\right\}
\end{equation}
where the $\bar{\Theta}_{r,s}$ takes the following form,
\begin{equation}
\bar{\Theta}_{r,s}(\tau):=\theta_{2}(\tau)^{4r}\theta_{3}(\tau)^{4s}+\theta_{2}(\tau)^{4s}\theta_{3}(\tau)^{4r}, \qquad r\leq s.
\end{equation}
We define $k$’th order modular differential operators $D^{(k)}_{q}$ as
\begin{equation}
D^{(k)}_{q}\chi(q):=\partial_{(2k-2)}\circ\dots\circ\partial_{(2)}\circ\partial_{(0)}\chi(q),
\end{equation}
then modular linear differential operators that are holomorphic and monic have the following generic form,
\begin{equation}
\mathcal{D}^{(k)}_{q}\equiv D^{(k)}_{q}+\sum_{r=1}^{k}f_{r}(q)D^{(k-r)}_{q}, \quad f_{r}(q)\in M_{2k}(\tilde{\Gamma},\mathbb{C})
\end{equation}
where $\tilde{\Gamma}$ denote  any congruence subgroup of $\Gamma$.

\subsection{$A^{(1)}_{1}$ at boundary admissible levels $k=-2+\frac{2}{u}$}
In this section, we give some specific MLDEs of irreducible $\mathbb{Z}_{2}$-twisted modules for $A^{(1)}_{1}$ at boundary admissible level $k=-2+\frac{2}{u}$. Following from Proposition \ref{admissible} and (\ref{formula59}), we have:
\begin{thm}
All  irreducible $\mathbb{Z}_{2}$-twisted modules of $L_{k}(\mathfrak{sl}_2)$ at admissible level in category $\mathcal{O}$ can be obtained by using $\ell=-\frac{1}{2}$ spectral flow on the untwisted modules in category $\mathcal{O}$. In particular, for boundary admissible level, all of  those irreducible twisted modules are ordinary modules. we find that the $q$-series characters satisfy the following relation,
\begin{equation}
	\mathrm{ch}[\sigma^{-\frac{1}{2}}(L(\Lambda_{k,j}))](q)=\mathrm{ch}[\sigma^{-\frac{1}{2}}(L(\Lambda_{k,u-1-j}))](q)
\end{equation}
 Furthermore, the number of independence $q$ series characters is $\frac{u+1}{2}$.
\end{thm}
Now let us give some concrete examples for small values of $u$.
\begin{ex}
Let us consider $\hat{\mathfrak{g}}=A^{(1)}_{1}$ at level $k=-\frac{4}{3}$,  the number of independence $q$-series characters is two, we denote these two characters of twisted modules as $\mathrm{ch}\left[\sigma^{-\frac{1}{2}}(L(\Lambda_{k,0}))\right]$ and $\mathrm{ch}\left[\sigma^{-\frac{1}{2}}(L(\Lambda_{k,1}))\right]$. They are annihilated
by a second-order $\Gamma^{0}(2)$-MLDE which we display here:
\begin{align}
\left(D^{(2)}_{q}-\frac{1}{96}\bar{\Theta}_{1,1}(\tau)\right)\mathrm{ch}\left[\sigma^{-\frac{1}{2}}(L(\Lambda_{k,i}))\right](q)=0\label{mlde1},\quad i=0,1.
\end{align}
Since the modular form  $M_{2k}(\Gamma^{0}(2))$ spanned by the $\Theta_{r,s}$ which can be rewritten in terms of twisted Eisenstein series, we can rewrite above $\mathrm{MLDE}$ (\ref{mlde1}) as
\begin{equation}
	\begin{aligned}
		&\left(D^{(2)}_{q}+\frac{4}{3}\mathbbm{E}_{4}\begin{bmatrix}-1\\1\end{bmatrix}+\frac{28}{3}\mathbbm{E}_{4}\begin{bmatrix}1\\-1\end{bmatrix}+\frac{28}{3}\mathbbm{E}_{4}\begin{bmatrix}-1\\-1\end{bmatrix}\right)\mathrm{ch}\left[\sigma^{-\frac{1}{2}}(L(\Lambda_{k,i}))\right](q)=0,\quad i=0,1.\\
	\end{aligned}
\end{equation}
\end{ex}
Actually, the twisted module character  $\mathrm{ch}\left[\sigma^{-\frac{1}{2}}(L(\Lambda_{k,1}))\right]\big|_{y=1}$ has two different physical interpretations. In \cite{Cordova:2017mhb}, the authors  computed the defect Schur indices of  the $(A_{1},A_{3})$ Argyres-Douglas theory,
\begin{align}
&\mathcal{I}_{\mathbb{S}}(q,x)=(q)^{2}_{\infty}\sum^{\infty}_{\substack{\ell_1,...,\ell_{3},\\k_1,...,k_{3}=0}}\f{(-1)^{\sum_{i=1}^{3}(k_i+\ell_i)}q^{\f12 \sum_{i=1}^{3}(k_i+\ell_i)+\ell_{2}(\ell_{1}+\ell_{3})}}{\prod_{i=1}^{3}(q)_{k_i}(q)_{\ell_i}} (x)^{\ell_1-k_1}\left(q^{\f{\ell_1-k_1}{2}}+q^{\f{k_1-\ell_1}{2}}-q^{\f{\ell_1+k_1}{2}}\right)\delta_{k_{2},\ell_{2}}\delta_{k_1+k_{3},\ell_{1}+\ell_{3}}.	
\end{align}
The corresponding VOA of the $(A_1, A_3)$ AD theory is just $L_{-\frac{4}{3}}(A_1^{(1)})$. One can check that the above surface defect index agrees with the twisted module character $\mathrm{ch}\left[\sigma^{-\frac{1}{2}}(L(\Lambda_{k,1}))\right]\big|_{y=1}$ with $x$ replaced by $\mathbf{z}^{2}$,
\begin{align}
\mathcal{I}_{\mathbb{S}}(q,\mathbf{z}^{2})=\text{ch}\left[\sigma^{-\frac{1}{2}}(L(\Lambda_{k,1}))\right]\big|_{y=1}.
\end{align}
In \cite{Fluder:2017oxm}, the authors compute the lens space index of the $(A_{1},A_{3})$ AD theory. For example, we have checked that  $\mathrm{ch}\left[\sigma^{-\frac{1}{2}}(L(\Lambda_{k,1}))\right]\big|_{y=1}$ agrees with the lens space index $\mathcal{I}^{\text{Mac}}_{(A_{1},D_{3})}\big|_{t=1}$ up to an overall factor by identify their "twisting parameter" with the spectral flow parameter. One  advantage of our expression is that operator spectrum of the 4d theory is much easier to read off, and modular properties are also apparent.
\begin{ex}
Let us consider $\hat{\mathfrak{g}}=A^{(1)}_{1}$ at level $k=-\frac{8}{5}$,  the number of independent $q$-series characters is three, we denote these three characters of twisted modules as $\mathrm{ch}[\sigma^{-\frac{1}{2}}(L(\Lambda_{k,j}))](q), j=0,1,2$. They satisfy a third-order $\Gamma^{0}(2)$-MLDE.
\begin{equation}
\left[D^{(3)}_{q}-\left(\frac{7}{450}\bar{\Theta}_{0,2}(\tau)+\frac{31}{1800}\bar{\Theta}_{1,1}(\tau)\right)D^{(1)}_{q}-\frac{1}{400}\bar{\Theta}_{1,2}(\tau)\right]\mathrm{ch}\left[\sigma^{-\frac{1}{2}}(L(\lambda_{k,j}))\right](q)=0, \quad j=0,1,2.
\end{equation}
\end{ex}
\begin{ex}
Let us consider $\hat{\mathfrak{g}}=A^{(1)}_{1}$ at level $k=-\frac{12}{7}$,  the number of independent $q$-series characters is four, we denote these three characters of twisted modules as $\mathrm{ch}[\sigma^{-\frac{1}{2}}(L(\Lambda_{k,j}))], j=0,1, \cdots, 3$. They satisfy a fourth-order $\Gamma^{0}(2)$-MLDE.
\begin{equation}
\begin{aligned}
& [D^{(4)}_{q}-\left(\frac{1}{18} \bar{\Theta}_{0,2}(\tau)+\frac{17}{1008}\bar{\Theta}_{1,1}(\tau)\right)D^{(2)}_{q}
\\
&+\left(\frac{50}{9261}\bar{\Theta}_{0,3}(\tau)-\frac{883}{49392}\bar{\Theta}_{1,2}(\tau)\right)D^{(1)}_{q}+\left(\frac{9}{10976}\bar{\Theta}_{1,3}(\tau)-\frac{225}{175616}\bar{\Theta}_{2,2}(\tau)\right)]\mathrm{ch}\left[\sigma^{-\frac{1}{2}}(L(\Lambda_{k,j}))\right](q)=0.
\end{aligned}
\end{equation}
\end{ex}

\subsection{$A^{(1)}_{1}$ at admissible level $k=-\frac{1}{2}$}
For $A^{(1)}_{1}$ at level $k=-\frac{1}{2}$, we can write down characters of admissible highest weight modules following \cite{Ridout:2008nh},
\begin{align}
&\text{ch}[\mathcal{L}_{0}]=\frac{\mathbf{y}^{-\frac{1}{2}}}{2}\left[\frac{\eta(\tau)}{\theta_{4}(\mathbf{z};q)}+\frac{\eta(\tau)}{\theta_{3}(\mathbf{z};q)}\right], \quad \text{ch}[\mathcal{D}^{+}_{-1/2}]=\frac{\mathbf{y}^{-\frac{1}{2}}}{2}\left[\frac{-i\eta(\tau)}{\theta_{1}(\mathbf{z};q)}+\frac{\eta(\tau)}{\theta_{2}(\mathbf{z};q)}\right]\\
&\text{ch}[\mathcal{L}_{1}]=\frac{\mathbf{y}^{-\frac{1}{2}}}{2}\left[\frac{\eta(\tau)}{\theta_{4}(\mathbf{z};q)}-\frac{\eta(\tau)}{\theta_{3}(\mathbf{z};q)}\right], \quad \text{ch}[\mathcal{D}^{+}_{-3/2}]=\frac{\mathbf{y}^{-\frac{1}{2}}}{2}\left[\frac{-i\eta(\tau)}{\theta_{1}(\mathbf{z};q)}-\frac{\eta(\tau)}{\theta_{2}(\mathbf{z};q)}\right]
\end{align}
 The characters $\text{ch}\left[\sigma^{-\frac{1}{2}}(M)\right]$ of associated $\mathbb{Z}_{2}$-twisted modules from $\ell=-\frac{1}{2}$ spectral flow are,
\begin{equation}
\begin{aligned}
&\mathrm{ch}\left[\sigma^{-\frac{1}{2}}(\mathcal{L}_{0})\right](\mathbf{y},\mathbf{z},q)=\frac{\mathbf{y}^{-\frac{1}{2}}\mathbf{z}^{\frac{1}{4}}q^{-\frac{1}{32}}}{2}\left[\frac{\eta(q\tau)}{\theta_{4}(\mathbf{z}q^{-\frac{1}{4}};q)}+\frac{\eta(\tau)}{\theta_{3}(\mathbf{z}q^{-\frac{1}{4}};q)}\right]\\
&\mathrm{ch}\left[\sigma^{-\frac{1}{2}}(\mathcal{L}_{1})\right](\mathbf{y},\mathbf{z},q)=\frac{\mathbf{y}^{-\frac{1}{2}}\mathbf{z}^{\frac{1}{4}}q^{-\frac{1}{32}}}{2}\left[\frac{\eta(\tau)}{\theta_{4}(\mathbf{z}q^{-\frac{1}{4}};q)}-\frac{\eta(\tau)}{\theta_{3}(\mathbf{z}q^{-\frac{1}{4}};q)}\right]\\
&\mathrm{ch}\left[\sigma^{-\frac{1}{2}}(\mathcal{D}^{+}_{-1/2})\right](\mathbf{y},\mathbf{z},q)=\frac{\mathbf{y}^{-\frac{1}{2}}\mathbf{z}^{\frac{1}{4}}q^{-\frac{1}{32}}}{2}\left[\frac{-i\eta(\tau)}{\theta_{1}(\mathbf{z}q^{-\frac{1}{4}};q)}+\frac{\eta(\tau)}{\theta_{2}(\mathbf{z}q^{-\frac{1}{4}};q)}\right]\\
&\mathrm{ch}\left[\sigma^{-\frac{1}{2}}(\mathcal{D}^{+}_{-3/2})\right](\mathbf{y},\mathbf{z},q)=\frac{\mathbf{y}^{-\frac{1}{2}}\mathbf{z}^{\frac{1}{4}}q^{-\frac{1}{32}}}{2}\left[\frac{-i\eta(\tau)}{\theta_{1}(\mathbf{z}q^{-\frac{1}{4}};q)}-\frac{\eta(\tau)}{\theta_{2}(\mathbf{z}q^{-\frac{1}{4}};q)}\right]
\end{aligned}
\end{equation}
For the character $\mathrm{ch}\left[\mathcal{L}_{0}\right]$ of vacuum module of  $L_{-1/2}(\mathfrak{sl}_{2})$, it is a solution of a third-order MLDE  under full $SL(2,\mathbb{Z})$ group,
\begin{equation}
\left[D^{(3)}_{q}-\frac{235}{4}	\mathbbm{E}_{4}(\tau)D^{(1)}_{q}-\frac{455}{8}\mathbbm{E}_{6}(\tau)\right]\mathrm{ch}[\mathcal{L}_{0}](q)=0
\end{equation}
There are two independent well-defined $q$-series characters of $\mathbb{Z}_{2}$-twisted modules denoted by $\mathrm{ch}[\sigma^{-\frac{1}{2}}(\mathcal{L}_{0})](q)$ and $\mathrm{ch}[\sigma^{-\frac{1}{2}}(\mathcal{L}_{1})](q)$. They satisfy a second-order $\Gamma^{0}(2)$ MLDE.
\begin{equation}
\left[D^{(2)}_{q}-\frac{5}{48}\bar{\Theta}_{0,1}(\tau)D^{(1)}_{q}+\left(\frac{25}{9216}\bar{\Theta}_{0,2}(\tau)-\frac{41}{9216}\bar{\Theta}_{1,1}(\tau)\right)\right]\mathrm{ch}[\sigma^{-\frac{1}{2}}(\mathcal{L}_{i})](q)=0,\quad i=0,1.
\end{equation}

\subsection{$A^{(1)}_{2}$ at boundary admissible level $k=-\frac{3}{2}$}

Consider the boundary principal admissible weight modules of $A^{(1)}_{2}$ at the boundary admissible level $k=-\frac{3}{2}$. There are four irreducible admissible highest weight modules of affine Lie algebra $A^{(1)}_{2}$, they are exactly the complete list of irreducible weak $L_{-3/2}(\mathfrak{sl}_{3})$ modules from category $\mathcal{O}$ \cite{Arakawa:2012xrk,pervse2008vertex}. Their characters can be written in terms of Jacobi theta function \cite{kac1988modular,kac2017remark}.
\begin{equation}
	\begin{aligned}
		&\text{ch}\left[\mathcal{L}\left(-\frac{3}{2}\Lambda_{0}\right)\right](\mathbf{y},\mathbf{z}_{1},\mathbf{z}_{2},q)=\mathbf{y}^{-\frac{3}{2}}\left(\frac{\eta(2\tau)}{\eta(\tau)}\right)^{-1}\frac{\theta_{1}\left(\mathbf{z}_{1};q^{2}\right)\theta_{1}\left(\mathbf{z}_{2};q^{2}\right)\theta_{1}\left(\mathbf{z}_{1}\mathbf{z}_{2};q^{2}\right)}{\theta_{1}\left(\mathbf{z}_{1};q\right)\theta_{1}\left(\mathbf{z}_{2};q\right)\theta_{1}\left(\mathbf{z}_{1}\mathbf{z}_{2};q\right)}\\
		&\text{ch}\left[\mathcal{L}\left(-\frac{3}{2}\Lambda_{1}\right)\right](\mathbf{y},\mathbf{z}_{1},\mathbf{z}_{2},q)=-\mathbf{y}^{-\frac{3}{2}}\left(\frac{\eta(2\tau)}{\eta(\tau)}\right)^{-1}\frac{\theta_{1}(\mathbf{z}_{2};q^{2})\theta_{4}(\mathbf{z}_{1};q^{2})\theta_{4}(\mathbf{z}_{1}\mathbf{z}_{2};q^{2})}{\theta_{1}(\mathbf{z}_{1};q)\theta_{1}(\mathbf{z}_{2};q)\theta_{1}(\mathbf{z}_{1}\mathbf{z}_{2};q)}\\
		&\text{ch}\left[\mathcal{L}\left(-\frac{3}{2}\Lambda_{2}\right)\right](\mathbf{y},\mathbf{z}_{1},\mathbf{z}_{2},q)=-\mathbf{y}^{-\frac{3}{2}}\left(\frac{\eta(2\tau)}{\eta(\tau)}\right)^{-1}\frac{\theta_{1}(\mathbf{z}_{1};q^{2})\theta_{4}(\mathbf{z}_{1}\mathbf{z}_{2};q^{2})\theta_{4}(\mathbf{z}_{2};q^{2})}{\theta_{1}(\mathbf{z}_{1};q)\theta_{1}(\mathbf{z}_{2};q)\theta_{1}(\mathbf{z}_{1}\mathbf{z}_{2};q)}\\
		&\text{ch}\left[\mathcal{L}\left(-\frac{\rho}{2}\right)\right](\mathbf{y},\mathbf{z}_{1},\mathbf{z}_{2},q)=\mathbf{y}^{-\frac{3}{2}}(\mathbf{z}_{1}\mathbf{z}_{2})^{\frac{3}{2}}q^{\frac{3}{2}}\left(\frac{\eta(2\tau)}{\eta(\tau)}\right)^{-1}\frac{\theta_{1}(\mathbf{z}_{2}^{-1}q^{-1};q^{2})\theta_{1}(\mathbf{z}_{1}^{-1}q^{-1};q^{2})\theta_{1}(\mathbf{z}_{1}^{-1}\mathbf{z}_{2}^{-1}q^{-2};q^{2})}{\theta_{1}(\mathbf{z}_{1};q)\theta_{1}(\mathbf{z}_{2};q)\theta_{1}(\mathbf{z}_{1}\mathbf{z}_{2};q)}
	\end{aligned}
\end{equation}
where $\Lambda_{i}$ denote the fundamental weights of affine Lie algebra $A^{(1)}_{2}$ and $\rho$ is the affine Weyl vector. Letting $z=\sum_{i=1}^{2}\mathfrak{z}_{i}\bar{\Lambda}_{i}$, where $\bar{\Lambda}_{i}$ are the fundamental weights of finite Lie algebra $\mathfrak{g}=\mathfrak{sl}_{3}$, we define $\mathbf{z}_{i}=e^{2\pi i \mathfrak{z}_{i}}$ which appeared in Jacobi theta function.

Now, we consider the action of spectral flow on these irreducible highest weight modules.
Firstly, we consider the  spectral flow along $\frac{1}{2}\bar{\Lambda}^{\vee}_{1}$ direction, and the character becomes,
\begin{equation}
	\text{ch}\left[\sigma^{\frac{1}{2}\Lambda_{1}^{\vee}}\left(\mathcal{L}\left(-\frac{3}{2}\Lambda_{0}\right)\right)\right](\mathbf{y},\mathbf{z}_{1},\mathbf{z}_{2},q)=(\mathbf{y}\mathbf{z}_{1}^{\frac{1}{3}}\mathbf{z}_{2}^{\frac{1}{6}}q^{\frac{1}{12}})^{-\frac{3}{2}}\left(\frac{\eta(2\tau)}{\eta(\tau)}\right)^{-1}\frac{\theta_{1}\left(\mathbf{z}_{1}q^{\frac{1}{2}};q^{2}\right)\theta_{1}\left(\mathbf{z}_{2};q^{2}\right)\theta_{1}\left(\mathbf{z}_{1}\mathbf{z}_{2}q^{\frac{1}{2}};q^{2}\right)}{\theta_{1}\left(\mathbf{z}_{1}q^{\frac{1}{2}};q\right)\theta_{1}\left(\mathbf{z}_{2};q\right)\theta_{1}\left(\mathbf{z}_{1}\mathbf{z}_{2}q^{\frac{1}{2}};q\right)}.
\end{equation}
This spectral flowed module is an ordinary module. It satisfies a second-order $\Gamma^{0}(2)$-modular linear differential equation,
\begin{equation}
	\left(D^{(2)}_{q}-\frac{5}{576}\bar{\Theta}_{0,2}(\tau)-\frac{11}{576}\bar{\Theta}_{1,1}(\tau)\right)\text{ch}\left[\sigma^{\frac{1}{2}\Lambda_{1}^{\vee}}\left(\mathcal{L}\left(-\frac{3}{2}\Lambda_{0}\right)\right)\right](q)=0.
\end{equation}
Secondly, consider the spectral flow of the character $\text{ch}\left[\mathcal{L}\left(-\frac{\rho}{2}\right)\right]$ along $\frac{1}{3}(\bar{\Lambda}^{\vee}_{1}+\bar{\Lambda}^{\vee}_{2})$ direction,
\begin{equation}\label{twist}
	\begin{aligned}
		\text{ch}\left[\sigma^{\frac{1}{3}\fwt{1}^{\vee}+\frac{1}{3}\fwt{2}^{\vee}}\left(\mathcal{L}\left(-\frac{\rho}{2}\right)\right)\right](y,\mathbf{z}_{1},\mathbf{z}_{2},q)&=(\mathbf{y}\mathbf{z}_{1}^{\frac{1}{3}}\mathbf{z}_{2}^{\frac{1}{3}}q^{\frac{1}{9}})^{-\frac{3}{2}}(\mathbf{z}_{1}\mathbf{z}_{2}q^{\frac{3}{2}})^{\frac{3}{2}}q^{\frac{3}{2}}\left(\frac{\eta(2\tau)}{\eta(\tau)}\right)^{-1}\\
		&\times\frac{\theta_{1}((\mathbf{z}_{2}q^{\frac{1}{3}})^{-1}q^{-1};q^{2})\theta_{1}((\mathbf{z}_{1}q^{\frac{1}{3}})^{-1}q^{-1};q^{2})\theta_{1}((\mathbf{z}_{1}\mathbf{z}_{2}q^{-\frac{2}{3}})^{-1}q^{-2};q^{2})}{\theta_{1}(\mathbf{z}_{1}q^{\frac{1}{3}};q)\theta_{1}(\mathbf{z}_{2}q^{\frac{1}{3}};q)\theta_{1}(\mathbf{z}_{1}\mathbf{z}_{2}q^{\frac{2}{3}};q)}
	\end{aligned}
\end{equation}
This spectral flowed character matches with the lens space index of $(A_1, D_4)$ AD theory \cite{Fluder:2017oxm} with suitable change of variables.

\subsection{Bershadsky-Polyakov Algebra $\mathrm{BP}^{k}$ with $k=-\frac{9}{4}$}
In previous examples, we consider the spectral flowed modules of affine Lie algebra $\hat{\mathfrak{g}}$. Now we consider an example of the affine $W$-algebra \cite{bouwknegt1995w,Feigin:1990pn,feigin1990representations,kac2003quantum}, the $W^{k}(\mathfrak{sl}_{3},f_{\mathrm{min}})$ which agrees with the $\mathrm{BP}^{k}$-algebra  defined in \cite{Bershadsky:1990bg}.  First, let us review the definition of $\text{BP}^{k}$-algebra.
\begin{dfn}
	\cite{Fehily:2020bif} Given $k\in\mathbb{C}$, $k\neq-3$, the level-k universal Bershadsky-Polyakov algebra $\text{BP}\,^{k}$ is the vertex operator algebra with vacuum $\bm{1}$ that is strongly and freely generated by fields $J(z)$, $G^{+}(z)$, $G^{-}(z)$ and $L(z)$ satisfying the complicated operator product expansions. The conformal weights of the generating fields $J(z)$, $G^{+}(z)$, $G^{-}(z)$ and $L(z)$ are $1$, $\frac{3}{2}$, $\frac{3}{2}$ and $2$ respectively, the central charge is,
	\begin{equation}
		c^{\mathrm{BP}}_{u,v}=-\frac{(2k+3)(3k+1)}{k+3}
	\end{equation}
\end{dfn}
The action of the spectral flow automorphism $\sigma^{\ell}$, $\ell\in\mathbb{Z}$ of the vertex algebra $\text{BP}\,^{k}$ on the modes of the generating field $J(z)$, $G^{+}(z)$, $G^{-}(z)$ and $L(z)$ is
\begin{equation}
	\begin{aligned}
		\sigma^{l}(J_{n})&=J_{n}-\frac{2k+3}{3}l\delta_{n,0}\bm{1},\\
		\sigma^{l}(G^{+}_{r})&=G^{+}_{r-l},\\
		\sigma^{l}(G^{-}_{r})&=G^{-}_{r+l},\\
		\sigma^{l}(L_{n})&=L_{n}-lJ_{n}+\frac{2k+3}{6}l^{2}\delta_{n,0}\bm{1}.
	\end{aligned}
\end{equation}
When $\ell$ is a half-integer, $\sigma^l$ exchanges  twisted and untwisted mode algebras \cite{Fehily:2020bif}.

If we consider the $|\lambda,\Delta\rangle\in M$ is a state of weight $\lambda$ and conformal dimension $\Delta$ for any module $M$, then the state $(\sigma^{\ell})^{*}|\lambda,\Delta\rangle\in(\sigma^{\ell})^{*}(M)$ satisfied,
\begin{equation}
	\begin{aligned}
		&h_{0}(\sigma^{\ell})^{*}|\lambda,\Delta\rangle=(\lambda+\ell\frac{2k+3}{3})(\sigma^{\ell})^{*}|\lambda,\Delta\rangle\\
		&L_{0}(\sigma^{\ell})^{*}|\lambda,\Delta\rangle=(\Delta+\ell\lambda+\frac{2k+3}{6}\ell^{2})(\sigma^{\ell})^{*}|\lambda,\Delta\rangle
	\end{aligned}
\end{equation}
For a $\mathrm{BP}^{k}$-algebra module $M$, the character has been defined in \cite{Fehily:2020bif},
\begin{equation}
	\text{ch}[M](\theta|\zeta|\tau)=\mathbf{y}^{\kappa}\text{tr}_{M}\left(\mathbf{z}^{J_{0}}q^{L_{0}-c^{\mathrm{BP}}_{u,v}/24}\right), \quad \kappa=\frac{2k+3}{6}
\end{equation}
where $\mathbf{y}=e^{2\pi i \theta}$, $\mathbf{z}=e^{2\pi i \zeta}$ and $q=e^{2\pi i \tau}$. The character of spectral flowed module $\sigma^{\ell}(M)$ for $\ell\in \frac{1}{2}\mathbb{Z}$ given by Lemma 4.3 in \cite{Fehily:2020bif},
\begin{equation}
	\mathrm{ch}[\sigma^{\ell}(M)]\left(\theta|\zeta|\tau\right)=\mathrm{ch}[M]\left(\theta+2\ell\zeta+\ell^{2}\tau|\zeta+\ell\tau|\tau\right)
\end{equation}
We consider the special case which the level $k=-\frac{9}{4}$ and $c^{\mathrm{BP}}_{u,v}=-\frac{23}{2}$, the highest weight vector $|\lambda,\Delta\rangle=|\frac{1}{4},-\frac{3}{8}\rangle$, and spectral flow parameter $\ell=\frac{1}{2}$. After spectral flow, the weight and conformal dimension become,
\begin{equation}
	\lambda^{'}=\frac{1}{4}+\frac{1}{2}\times (-\frac{1}{2})=0, \qquad \Delta'=-\frac{3}{8}+\frac{1}{2}\times \frac{1}{4}-\frac{1}{4}\times \frac{1}{4}=-\frac{5}{16}
\end{equation}
According to \cite{Cordova:2017mhb}, the character of spectral flowed module can be written as follows,
\begin{equation}
	\text{ch}[\sigma^{\frac{1}{2}}(M_{|\frac{1}{4},-\frac{3}{8}\rangle})](q)=q^{\frac{1}{6}}\left(1+4q+10q^{2}+24q^{3}+51q^{4}+100q^{5}+\mathcal{O}(q^{6})\right)
\end{equation}
This character satisfies a third-order MLDE  under full $SL(2,\mathbb{Z})$ group,
\begin{equation}
	\left[D^{(3)}_{q}-25	\mathbbm{E}_{4}(\tau)D^{(1)}_{q}-175\mathbbm{E}_{6}(\tau)\right]\text{ch}[\sigma^{\frac{1}{2}}(M_{|\frac{1}{4},-\frac{3}{8}\rangle})](q)=0
\end{equation}

\section{$\mathfrak{d}_{4}$ with non-admissible level $k=-2$}\label{section6}
In this section, we propose a relation between simple modules and spectral flowed modules of the non-admissible affine vertex algebra $\mathcal{L}_{\mathfrak{d}_{4}}(-2\Lambda_{0})$. We also find an ordinary module  $\sigma^{\frac{1}{2}\bar{\Lambda}_2}(\mathcal{L}_{\mathfrak{d}_{4}}(-\Lambda_{2}))$, whose character satisfies a second-order $\Gamma^{0}(2)$ MLDE.

Firstly, recall the result for the simple modules of $\mathcal{L}_{\mathfrak{d}_{4}}(-2\Lambda_{0})$.
\begin{thm}
\cite{pervse2013note} The set $\{\mathcal{L}_{\mathfrak{d}_{4}}(-2\Lambda_{0}),\mathcal{L}_{\mathfrak{d}_{4}}(-2\Lambda_{1}),\mathcal{L}_{\mathfrak{d}_{4}}(-2\Lambda_{3}),\mathcal{L}_{\mathfrak{d}_{4}}(-2\Lambda_{4}),\mathcal{L}_{\mathfrak{d}_{4}}(-\Lambda_{2})\}$ provides a complete list of irreducible weak $\mathcal{L}_{\mathfrak{d}_{4}}(-2\Lambda_{0})$-modules from the Category $\mathcal{O}.$ $\mathcal{L}_{\mathfrak{d}_{4}}(-2\Lambda_{0})$ is the unique irreducible ordinary module for $\mathcal{L}_{\mathfrak{d}_{4}}(-2\Lambda_{0})$. And every ordinary $\mathcal{L}_{\mathfrak{d}_{4}}(-2\Lambda_{0})$-module is completely reducible.
\end{thm}

We consider the spectral flow automorphsim that act on this vacuum character. The fundamental weights $\bar{\Lambda}_{i}$ of finite Lie algebra $\mathfrak{d}_{4}$ can be written in terms of the linear combination of simple roots $\alpha_{i}$.
\begin{equation}
\begin{aligned}
&\bar{\Lambda}_{1}=\alpha_{1}+\alpha_{2}+\frac{1}{2}\alpha_{3}+\frac{1}{2}\alpha_{4}, \quad \bar{\Lambda}_{2}=\alpha_{1}+2\alpha_{2}+\alpha_{3}+\alpha_{4}\\
&\bar{\Lambda}_{3}=\frac{1}{2}\alpha_{1}+\alpha_{2}+\alpha_{3}+\frac{1}{2}\alpha_{4}, \quad \bar{\Lambda}_{4}=\frac{1}{2}\alpha_{1}+\alpha_{2}+\frac{1}{2}\alpha_{3}+\alpha_{4}
\end{aligned}
\end{equation}
Since the $\mathfrak{d}_{4}$ is simple-laced, then $\alpha_{i}=\alpha^{\vee}_{i}$ and $\bar{\Lambda}_{i}=\bar{\Lambda}^{\vee}_{i}$, we will use roots (weights) or coroots (coweights) without distinction. The highest root of finite Lie algebra $\mathfrak{d}_{4}$ is $\theta=\bar{\Lambda}_{2}$, therefore the marks and comarks of affine Lie algebra $\hat{\mathfrak{d}}_{4}$ are $(a_{i})=(a_{i}^{\vee})=(1,1,2,1,1)$, the level of affine weight $\Lambda=\sum_{i=0}^{4}\lambda_{i}\Lambda_{i}$ is then given by,
\begin{equation}
k=\lambda_{0}+\lambda_{1}+2\lambda_{2}+\lambda_{3}+\lambda_{4}
\end{equation}

\begin{thm}
The  $\mathcal{L}_{\mathfrak{d}_{4}}(-2\Lambda_{i})$ with $i=1,3,4$ is the same as spectral flowed  module $\sigma^{\bar{\Lambda}_{i}}(\mathcal{L}_{\mathfrak{d}_{4}}(-2\Lambda_{0}))$, respectively.
\end{thm}
\begin{proof}
The powers of $\tau_{i}$ acts as follows,
\begin{equation}
	\tau_{i}^{\ell}(h^{j}_{n})=h^{j}_{n}-\ell(\alpha_{i},\alpha_{i})\delta_{n,0}K
\end{equation}
We can compute the action of this automorphsim along $\bar{\Lambda}_{1}$ on $h^{i}_{0}$,
\begin{equation}
	\tau_{1}\tau_{2}\tau_{3}^{\frac{1}{2}}\tau_{4}^{\frac{1}{2}}(h^{1}_{0})=h^{1}_{0}+K, \quad \tau_{1}\tau_{2}\tau_{3}^{\frac{1}{2}}\tau_{4}^{\frac{1}{2}}(h^{2}_{0})=h^{2}_{0}
\end{equation}
\begin{equation}
	\tau_{1}\tau_{2}\tau_{3}^{\frac{1}{2}}\tau_{4}^{\frac{1}{2}}(h^{3}_{0})=h^{0}_{3}, \quad \tau_{1}\tau_{2}\tau_{3}^{\frac{1}{2}}\tau_{4}^{\frac{1}{2}}(h^{4}_{0})=h^{4}_{0}
\end{equation}
That means the spectral flow automorphsim change highest weight $-2\Lambda_{0}$ into another highest weight $-2\Lambda_{1}$. One can get highest weight $-2\Lambda_{3}$ and $-2\Lambda_{4}$ in the same way. According to cite \cite{dong1996simple}, the $\bar{\Lambda}_{i}$ direction preserves the irreduciblity of the module $\mathcal{L}_{\mathfrak{d}_{4}}(-2\Lambda_{i})$\footnote{We would like to thank Tomoyuki Arakawa and Kazuya Kawasetsu for pointing out this fact.}, this completes the proof.
\end{proof}

Actually, the closed form expressions of characters of all simple modules of $\mathcal{L}_{\mathfrak{d}_{4}}(-2\Lambda_{0})$ are conjectured from the SCFT/VOA correspondence as \cite{Pan:2021mrw, Zheng:2022zkm}
\begin{equation}
\begin{aligned}
&\text{ch}[\mathcal{L}_{\mathfrak{d}_{4}}(-2\Lambda_{0})]=\mathcal{I}_{0,4}\\
&\text{ch}[\mathcal{L}_{\mathfrak{d}_{4}}(-2\Lambda_{1})]=\mathcal{I}_{0,4}-2R_{1}\\
&\text{ch}[\mathcal{L}_{\mathfrak{d}_{4}}(-\Lambda_{2})]=-2\mathcal{I}_{0,4}+2R_{1}+2R_{2}\\
&\text{ch}[\mathcal{L}_{\mathfrak{d}_{4}}(-2\Lambda_{3})]=\mathcal{I}_{0,4}-R_{1}-R_{2}-R_{3}-R_{4}\\
&\text{ch}[\mathcal{L}_{\mathfrak{d}_{4}}(-2\Lambda_{4})]=\mathcal{I}_{0,4}-R_{1}-R_{2}-R_{3}+R_{4}
\end{aligned}
\end{equation}
where $\mathcal{I}_{0,4}$ is the Schur index of an $SU(2)$ gauge theory with four hypermultiplets, and the $R_j$ functions are
\begin{equation}
R_{j}(\widetilde{\bm{m}}_{i},\tau)=\frac{i}{2}\frac{\theta_{1}(2\widetilde{\bm{m}}_{j})}{\eta(\tau)}\prod_{l\neq j}\frac{\eta(\tau)}{\theta_{1}(\widetilde{\bm{m}}_{j}+\widetilde{\bm{m}}_{l})}\frac{\eta(\tau)}{\theta_{1}(\widetilde{\bm{m}}_{j}-\widetilde{\bm{m}}_{l})},\quad j=1,2,3,4.
\end{equation}
Here $\widetilde{m}_{i}=e^{2\pi i \bm{\widetilde{m}}_{i}}$ are related to the Cartan element $\widetilde{z}_{i}=e^{2\pi i \bm{\widetilde{z}}_{i}}$ of $\mathfrak{d}_{4}$ as follows,
\begin{align}
\widetilde{z}_{1}=\frac{\widetilde{m}_{1}}{\widetilde{m}_{2}}, \qquad \widetilde{z}_{2}=\frac{\widetilde{m}_{2}}{\widetilde{m}_{3}},
\qquad \widetilde{z}_{3}=\widetilde{m}_{3}\widetilde{m}_{4},
\qquad
\widetilde{z}_{4}=\frac{\widetilde{m}_{3}}{\widetilde{m}_{4}}
\end{align}
For our purpose, we give $\text{ch}[\mathcal{L}_{\mathfrak{d}_{4}}(-2\Lambda_{0})]$, $\mathrm{ch}[\sigma^{\bar{\Lambda}_{1}}(\mathcal{L}_{\mathfrak{d}_{4}}(-2\Lambda_{0}))]$, $\mathrm{ch}[\sigma^{\bar{\Lambda}_{3}}(\mathcal{L}_{\mathfrak{d}_{4}}(-2\Lambda_{0}))]$, $\mathrm{ch}[\sigma^{\bar{\Lambda}_{4}}(\mathcal{L}_{\mathfrak{d}_{4}}(-2\Lambda_{0}))]$ explicitly,
\begin{equation}
\begin{aligned}
\text{ch}[\mathcal{L}_{\mathfrak{d}_{4}}(-2\Lambda_{0})](q;,\widetilde{z}_{1},\widetilde{z}_{2},\widetilde{z}_{3},\widetilde{z}_{4})&=\frac{1}{2}\frac{\eta(\tau)^{2}}{\theta_{1}(\tilde{z}_{1}\tilde{z}_{2}^{2}\tilde{z}_{3}\tilde{z}_{4};q)\theta_{1}(\tilde{z}_{1};q)\theta_{1}(\tilde{z}_{4};q)\theta_{1}(\tilde{z}_{3};q)}\\&\times\sum_{\vec{\alpha}=\pm}\left(\prod_{i=1}^{4}\alpha_{i}\right)E_{2}\begin{bmatrix}1\\\left(\widetilde{z}_{1}^{\frac{1}{2}}\widetilde{z}_{2}\widetilde{z}_{3}^{\frac{1}{2}}\widetilde{z}_{4}^{\frac{1}{2}}\right)^{\alpha_{1}}\left(\widetilde{z}_{1}^{\frac{1}{2}}\right)^{\alpha_{2}}\left(\widetilde{z}_{4}^{\frac{1}{2}}\right)^{\alpha_{3}}\left(\widetilde{z}_{3}^{\frac{1}{2}}\right)^{\alpha_{4}}
\end{bmatrix}\\
\mathrm{ch}[\sigma^{\bar{\Lambda}_{1}}(\mathcal{L}_{\mathfrak{d}_{4}}(-2\Lambda_{0}))](q;,\widetilde{z}_{1},\widetilde{z}_{2},\widetilde{z}_{3},\widetilde{z}_{4})&=\frac{1}{2}\left(y\widetilde{z}_{1}\widetilde{z}_{2}\widetilde{z}_{3}^{\frac{1}{2}}\widetilde{z}_{4}^{\frac{1}{2}}q^{\frac{1}{2}}\right)^{-2}\frac{\eta(\tau)^{2}}{\theta_{1}(\widetilde{z}_{1}q\widetilde{z}_{2}^{2}\widetilde{z}_{3}\widetilde{z}_{4};q)\theta_{1}(\widetilde{z}_{1}q;q)\theta_{1}(\widetilde{z}_{4};q)\theta_{1}(\widetilde{z}_{3};q)}\\
&\times\sum_{\vec{\alpha}=\pm}\left(\prod_{i=1}^{4}\alpha_{i}\right)E_{2}\begin{bmatrix}1\\\left(\widetilde{z}_{1}^{\frac{1}{2}}q^{\frac{1}{2}}\widetilde{z}_{2}\widetilde{z}_{3}^{\frac{1}{2}}\widetilde{z}_{4}^{\frac{1}{2}}\right)^{\alpha_{1}}\left(\widetilde{z}_{1}^{\frac{1}{2}}q^{\frac{1}{2}}\right)^{\alpha_{2}}\left(\widetilde{z}_{4}^{\frac{1}{2}}\right)^{\alpha_{3}}\left(\widetilde{z}_{3}^{\frac{1}{2}}\right)^{\alpha_{4}}
\end{bmatrix}\\
\mathrm{ch}[\sigma^{\bar{\Lambda}_{3}}(\mathcal{L}_{\mathfrak{d}_{4}}(-2\Lambda_{0}))](q;,\widetilde{z}_{1},\widetilde{z}_{2},\widetilde{z}_{3},\widetilde{z}_{4})&=\frac{1}{2}\left(y\widetilde{z}_{1}^{\frac{1}{2}}\widetilde{z}_{2}\widetilde{z}_{3}\widetilde{z}_{4}^{\frac{1}{2}}q^{\frac{1}{2}}\right)^{-2}\frac{\eta(\tau)^{2}}{\theta_{1}(\widetilde{z}_{1}\widetilde{z}_{2}^{2}\widetilde{z}_{3}q\widetilde{z}_{4};q)\theta_{1}(\widetilde{z}_{1};q)\theta_{1}(\widetilde{z}_{4};q)\theta_{1}(\widetilde{z}_{3}q;q)}\\
&\times\sum_{\vec{\alpha}=\pm}\left(\prod_{i=1}^{4}\alpha_{i}\right)E_{2}\begin{bmatrix}1\\\left(\widetilde{z}_{1}^{\frac{1}{2}}\widetilde{z}_{2}\widetilde{z}_{3}^{\frac{1}{2}}q^{\frac{1}{2}}\widetilde{z}_{4}^{\frac{1}{2}}\right)^{\alpha_{1}}\left(\widetilde{z}_{1}^{\frac{1}{2}}\right)^{\alpha_{2}}\left(\widetilde{z}_{4}^{\frac{1}{2}}\right)^{\alpha_{3}}\left(\widetilde{z}_{3}^{\frac{1}{2}}q^{\frac{1}{2}}\right)^{\alpha_{4}}
\end{bmatrix}\\
\mathrm{ch}[\sigma^{\bar{\Lambda}_{4}}(\mathcal{L}_{\mathfrak{d}_{4}}(-2\Lambda_{0}))](q;,\widetilde{z}_{1},\widetilde{z}_{2},\widetilde{z}_{3},\widetilde{z}_{4})&=\frac{1}{2}\left(y\widetilde{z}_{1}^{\frac{1}{2}}\widetilde{z}_{2}\widetilde{z}_{3}^{\frac{1}{2}}\widetilde{z}_{4}q^{\frac{1}{2}}\right)^{-2}\frac{\eta(\tau)^{2}}{\theta_{1}(\widetilde{z}_{1}\widetilde{z}_{2}^{2}\widetilde{z}_{3}\widetilde{z}_{4}q;q)\theta_{1}(\widetilde{z}_{1};q)\theta_{1}(\widetilde{z}_{4}q;q)\theta_{1}(\widetilde{z}_{3};q)}\\
&\times\sum_{\vec{\alpha}=\pm}\left(\prod_{i=1}^{4}\alpha_{i}\right)E_{2}\begin{bmatrix}1\\\left(\widetilde{z}_{1}^{\frac{1}{2}}\widetilde{z}_{2}\widetilde{z}_{3}^{\frac{1}{2}}\widetilde{z}_{4}^{\frac{1}{2}}q^{\frac{1}{2}}\right)^{\alpha_{1}}\left(\widetilde{z}_{1}^{\frac{1}{2}}\right)^{\alpha_{2}}\left(\widetilde{z}_{4}^{\frac{1}{2}}q^{\frac{1}{2}}\right)^{\alpha_{3}}\left(\widetilde{z}_{3}^{\frac{1}{2}}\right)^{\alpha_{4}}
\end{bmatrix}
\end{aligned}
\end{equation}
then, one can show that these expressions satisfy the following relations
\begin{equation}
\mathrm{ch}[\sigma^{\bar{\Lambda}_{i}}(\mathcal{L}_{\mathfrak{d}_{4}}(-2\Lambda_{0}))] = \mathrm{ch}\left[\mathcal{L}_{\mathfrak{d}_{4}}(-2\Lambda_i)\right],\quad i=1,3,4.
\end{equation}
It means that we prove the closed form expressions of characters of some simple modules of $\mathcal{L}_{\mathfrak{d}_{4}}(-2\Lambda_{0})$ in terms of $R_{j}$ function.
We can also prove another relation between characters of $\mathcal{L}_{\mathfrak{d}_{4}}(-2\Lambda_{0})$ modules and the partition function of the curved $\beta\gamma$ system \cite{Eager:2019zrc}. The partition function of the curved $\beta\gamma$ system on complex Grassmannian $\text{Gr}(2,4)$ is given by,
\begin{equation}
Z_{\mathfrak{d}_{4}}(t,\bm{m}^{\mathfrak{a}_{3}}_{3},\tau)=\frac{i\eta(\tau)\theta_{1}(2\sigma,\tau)}{\prod_{\omega\in \rho}\theta_{1}(\sigma+(\bm{m}^{\mathfrak{a}_{3}},\omega),\tau)},
\end{equation}
where $t=e^{2\pi i \sigma}$ and $q=e^{2 \pi i \tau}$ as usual. The product in the denominator is over the weights in the  representation $\rho$ of  $\mathfrak{a}_{3}$ with highest weight $\bar{\Lambda}_{\mathfrak{a}_{3}}$. The authors of \cite{Eager:2019zrc} found that this partition function is also given by
\begin{equation}
Z_{\mathfrak{d}_{4}}(t,\bm{m}^{\mathfrak{a}_{3}}_{3},\tau)=\text{ch}[\mathcal{L}_{\mathfrak{d}_{4}}(-2\Lambda_{0})](\bm{m}^{\mathfrak{d}_{4}},\tau)-\text{ch}[\mathcal{L}_{\mathfrak{d}_{4}}(-2\Lambda_{4})](\bm{m}^{\mathfrak{d}_{4}},\tau)
\end{equation}
with the following identifications of parameters
\begin{equation}
\bm{m}^{\mathfrak{d}_{4}}_{i}=\bm{m}^{\mathfrak{a}_{3}}_{i}, \quad \text{for}\,\, i=1,2,3, \quad \bm{m}^{\mathfrak{d}_{4}}_{4}=\sigma-\frac{\bm{m}^{\mathfrak{a}_{3}}_{1}}{2}-\bm{m}^{\mathfrak{a}_{3}}_{2}-\frac{\bm{m}^{\mathfrak{a}_{3}}_{3}}{2}.
\end{equation}
Using the expression of $\text{ch}[\mathcal{L}_{\mathfrak{d}_{4}}(-2\Lambda_0)]$, one also sees that
\begin{equation}
Z_{\mathfrak{d}_{4}}(t,\bm{m}^{\mathfrak{a}_{3}}_{3},\tau)=\text{ch}[\mathcal{L}_{\mathfrak{d}_{4}}(-2\Lambda_{0})](\bm{m}^{\mathfrak{d}_{4}},\tau)-\text{ch}[\sigma^{\bar{\Lambda}_4}(\mathcal{L}_{\mathfrak{d}_{4}}(-2\Lambda_{0}))](\bm{m}^{\mathfrak{d}_{4}},\tau).
\end{equation}
Therefore we have
\begin{equation}
\text{ch}[\mathcal{L}_{\mathfrak{d}_{4}}(-2\Lambda_{4})](\bm{m}^{\mathfrak{d}_{4}},\tau)
=\text{ch}[\sigma^{\bar{\Lambda}_4}(\mathcal{L}_{\mathfrak{d}_{4}}(-2\Lambda_{0}))](\bm{m}^{\mathfrak{d}_{4}},\tau).
\end{equation}
Triality gives similar results for $i=1$ and $2$.

Now, let us consider the character of the spectral flowed module of $\mathcal{L}_{\mathfrak{d}_{4}}(-\Lambda_{2})$ along $-\frac{1}{2}\bar{\Lambda}_2$ direction,
\begin{equation}
\text{ch}[\sigma^{-\frac{1}{2}\bar{\Lambda}_2}(\mathcal{L}_{\mathfrak{d}_{4}}(-\Lambda_{2}))]=\left(y\tilde{z}_{1}^{\frac{1}{2}}\tilde{z}_{2}\tilde{z}_{3}^{\frac{1}{2}}\tilde{z}_{4}^{\frac{1}{2}}q^{\frac{1}{4}}\right)^{-2} \text{ch}[\mathcal{L}_{\mathfrak{d}_{4}}(-\Lambda_{2})]\left(q;\tilde{z}_{1},\tilde{z}_{2}q^{-\frac{1}{2}},\tilde{z}_{3},\tilde{z}_{4}\right).
\end{equation}
$\text{ch}[\sigma^{-\frac{1}{2}\bar{\Lambda}_2} (\mathcal{L}_{\mathfrak{d}_{4}}(-\Lambda_{2}))]$ is the character of an ordinary module as it converges under the limit $\tilde{z}_i\rightarrow 1$. It equals to the defect index $\mathcal{I}^{\text{defet}}_{0,4}(k=1)$ of \cite{Zheng:2022zkm}, which satisfies the following second-order $\Gamma^{0}(2)$ MLDE
\begin{equation}
\left(D_{q}^{(2)}+\frac{1}{144}\bar{\Theta}_{0,2}(\tau)-\frac{37}{288}\bar{\Theta}_{1,1}(\tau)\right) \text{ch}[\sigma^{\frac{1}{2}\bar{\Lambda}_2}(\mathcal{L}_{\mathfrak{d}_{4}}(-\Lambda_{2}))]\big|_{\tilde{z}_{i}\rightarrow 1}=0,
\end{equation}
without $D_{q}^{(1)}$ term.

\appendix
\section{Theta functions}

We first summarize some basic facts about affine Lie of type $A^{(1)}_{1}$.

\begin{table}[ht]
\centering
\begin{tabular}{|l|l|}
    \hline
     Cartan matrix &    $\begin{pmatrix}2 & -2 \\-2 & 2 \end{pmatrix} $                \\
    \hline
     Simple roots  &    $\{\alpha_0, \alpha_1\} $                                       \\
    \hline
$\mathfrak{h}^*$   &  $\text{Span}_{\mathbb{C}}\{\alpha_0, \alpha_1, \Lambda_0 \}$       \\
     \hline
     Bilinear form on  $\mathfrak{h}^*$ &  $\langle \alpha_0, \alpha_0 \rangle=2$       \\
                                        &  $\langle \alpha_1, \alpha_1 \rangle=2$       \\
                                        & $\langle \alpha_1, \alpha_0 \rangle=\langle \al_0, \al_1\rangle=-2$   \\
                                        & $\langle \al_0, \Lambda_0 \rangle=1$               \\
                                        & $\langle \Lambda_0, \Lambda_0 \rangle=\langle \al_1, \Lambda_0 \rangle=0$  \\

     \hline
      basic imaginary root             & $\delta=\al_0+\al_1$         \\
     \hline
     $\mathfrak{h}$ &  $\text{Span}_{\mathbb{C}}\{\alpha_0^{\vee}=h_1, \alpha_2^{\vee}=h_2, d \}$   \\
    \hline
     Central element &         $c=\al_0^{\vee}+\al_1^{\vee}$                         \\
    \hline
    Fundamental weights &    $\{\Lambda_{0},\Lambda_1\}$                              \\
                        &    $\langle \Lambda_{i}, \alpha_j^{\vee}\rangle=\delta_{i,j} $,  $\Lambda_{i}(d)=0$, $i=0,1$ \\
    \hline
    Lattice $M$        &  $ \mathbb{Z} h_1$            \\
    \hline
    Lattice $M^*$      & $\frac12 \mathbb{Z} \al_1$            \\
    \hline
    Integral forms $P$        & $\{\lambda\in \mathfrak{h}^*|\langle \lambda(\al_i^{\vee})\in \mathbb{Z}, i=0,1\}$\\
    \hline
    Positive integral forms $P_{+}$ & $\{\lambda\in \mathfrak{h}^*|\langle \lambda(\al_i^{\vee})\in \mathbb{Z}_{\geq 0}, i=0,1\}$\\
    \hline
    Weyl group         &  $t(M)\rtimes \overline{W}$, $t(M)=\{t_m|m\in M\}$, $\overline{W}=\{s_{1}\}$   \\

    \hline
    Lacing number     & 1   \\
    \hline
    Coxeter dual number & 2 \\
    \hline
\end{tabular}
\end{table}

One can use classical theta functions to define {\em Jacobi theta functions} of degree two:
\begin{align}\label{trfmt7}
\begin{split}
    &\theta_{3}(\mathbf{z};q)\equiv\vartheta_{00}(\tau,z)=\Theta_{2,2}(\tau,z)+\Theta_{0,2}(\tau,z)=\displaystyle\sum_{n\in\mathbb{Z}}\mathbf{z}^{2(n+\frac12)}q^{2(n+\frac12)^2}+\mathbf{z}^{2n}q^{2n^2}=\sum_{n\in\mathbb{Z}}\mathbf{z}^{n}q^{\frac{n^2}{2}},\\
    &\theta_{4}(\mathbf{z};q)\equiv\vartheta_{01}(\tau,z)=-\Theta_{2,2}(\tau,z)+\Theta_{0,2}(\tau,z)=\displaystyle\sum_{n\in\mathbb{Z}}-\mathbf{z}^{2(n+\frac12)}q^{2(n+\frac12)^2}+\mathbf{z}^{2n}q^{2n^2}=\sum_{n\in \mathbb{Z}}(-1)^{n}\mathbf{z}^nq^{\frac{n^2}{2}},\\
    &\theta_{2}(\mathbf{z};q)\equiv\vartheta_{10}(\tau,z)=\Theta_{1,2}(\tau,z)+\Theta_{-1,2}(\tau,z)=\displaystyle\sum_{n\in \mathbb{Z}}\mathbf{z}^{2(n+\frac14)}q^{2(n+\frac14)^2}+\mathbf{z}^{2(n-\frac14)}q^{2(n-\frac14)^2}=\sum_{n\in \mathbb{Z}}\mathbf{z}^{(n+\frac12)}q^{\frac{(n+\frac12)^2}{2}},\\
    &\theta_{1}(\mathbf{z};q)\equiv-\vartheta_{11}(\tau,z)=i\Theta_{1,2}(\tau,z)-i\Theta_{-1,2}(\tau,z)=\displaystyle\sum_{n\in \mathbb{Z}}i\mathbf{z}^{2(n+\frac14)}q^{2(n+\frac14)^2}-i\mathbf{z}^{2(n-\frac14)}q^{2(n-\frac14)^2}\\
    &\quad\quad\quad\quad\quad\quad\quad\quad\quad\quad\quad\quad\quad\quad\quad\quad\quad\quad\quad\quad\quad=\sum_{n\in\mathbb{Z}}e^{\pi i (n+\frac12)}\mathbf{z}^{(n+\frac12)}q^{\frac{(n+\frac12)^2}{2}}.
\end{split}
\end{align}

Using $\eta(q)=q^{\frac{1}{24}}\prod_{n=1}^{\infty}(1-q^n)$, one can gets the following infinite product identities
\begin{align}\label{trfmt9}
\begin{split}
  &\displaystyle\prod_{n=1}^{\infty}(1+\mathbf{z}q^{n-\frac12})(1+\mathbf{z}^{-1}q^{n-\frac12})=q^{\frac{1}{24}}\frac{\vartheta_{00}(\tau,z)}{\eta(q)},\\
  &\displaystyle\prod_{n=1}^{\infty}(1-\mathbf{z}q^{n-\frac12})(1-\mathbf{z}^{-1}q^{n-\frac12})=q^{\frac{1}{24}}\frac{\vartheta_{01}(\tau,z)}{\eta(q)},\\
  &\displaystyle\prod^{\infty}_{n=1}(1+\mathbf{z}q^n)(1+\mathbf{z}^{-1}q^{n-1})=q^{-\frac{1}{12}}\mathbf{z}^{-\frac12}\frac{\vartheta_{10}(\tau,z)}{\eta(q)},\\
  &\displaystyle\prod^{\infty}_{n=1}(1-\mathbf{z}q^n)(1-\mathbf{z}^{-1}q^{n-1})=-iq^{-\frac{1}{12}}\mathbf{z}^{-\frac12}\frac{\vartheta_{11}(\tau,z)}{\eta(q)},
\end{split}
\end{align}
where $q$ is as always $e^{2\pi i \tau}.$

\section{Proof in Proposition \ref{bimod}}
The calculation of (\ref{appen1}):
\begin{align*}
   &\ju T_-^a T_0^b T_+^d P(E_1(n,k)) \\
  & =  T_-^a T_0^b T_+^d \p{\prod_{r=1}^{n}\prod_{s=1}^{k-1}G_{-r-st}}T_+^n\\
  & = T_-^a\p{\prod_{r=1}^{n}\prod_{s=1}^{k-1}G_{-r-st-d}}T_0^b T_+^{d+n}\\
   & =
   T_-^a\p{\prod_{r=1}^{n}\prod_{s=1}^{k-1}(T_-T_+-(-r-st-d)T_0+(-r-st-d)(-r-st-d+1))}T_0^b T_+^{d+n}\\
   & =
   T_-^a\p{\prod_{r=1}^{n}\prod_{s=1}^{k-1}(r+st+d)(T_0+r+st+d-1)}T_0^bT_+^{n+d} \ju \mod T_-U(L_0).
\end{align*}

The calculation of (\ref{appen2}):
\begin{align*}
   & \ju T_-^a T_0^b T_+^d P(E_2(n,k)) \\
   &= T_-^a T_0^b T_+^m \p{\prod_{i=1}^{p-n}G_{-i}}\prod_{r=0}^{p-n-1}\prod_{s=1}^{q-k}G_{r+st-p+n}\\
   &= T_-^a T_0^b T_+^m \prod_{r=0}^{p-n-1}\prod_{s=0}^{q-k}G_{r+st-p+n}\\
   &= T_-^a \prod_{r=0}^{p-n-1}\prod_{s=0}^{q-k}G_{r+st-p+n-m}T_0^b T_+^m\\
   &=T_-^a\prod_{r=1}^{p-n}\prod_{s=0}^{q-k}G_{st-m-r}T_0^b T_+^m\\
   &=T_-^a\prod_{r=1}^{p-n}\prod_{s=0}^{q-k}(st-m-r)(-T_0+st-m-r+1)T_0^b T_+^m  \ju \mod T_-U(L_0).\\
\end{align*}
\bibliographystyle{alpha}

\bibliography{Fusion}

\newcommand{\etalchar}[1]{$^{#1}$}
\begin{thebibliography}{BPRvR15}

\bibitem[AD95]{Argyres:1995jj}
Philip~C. Argyres and Michael~R. Douglas.
\newblock {New phenomena in SU(3) supersymmetric gauge theory}.
\newblock {\em Nucl. Phys. B}, 448:93--126, 1995.

\bibitem[AK18]{Arakawa:2016hkg}
Tomoyuki Arakawa and Kazuya Kawasetsu.
\newblock Quasi-lisse vertex algebras and modular linear differential
  equations.
\newblock {\em Lie Groups, Geometry, and Representation Theory: A Tribute to
  the Life and Work of Bertram Kostant}, pages 41--57, 2018.

\bibitem[AM95]{adamovic1995vertex}
Drazen Adamovic and Antun Milas.
\newblock {Vertex operator algebras associated to modular invariant
  representations for $A^{(1)}_{1}$}.
\newblock {\em Mathematical Research Letters}, 2:563--575, 1995.

\bibitem[AMP21]{adamovic2021certain}
Drazen Adamovic, Antun Milas, and Michael Penn.
\newblock {On certain W-algebras of type $W_{k}(sl(4), f)$}.
\newblock {\em Contemporary Math}, 2021.

\bibitem[Ara16]{Arakawa:2012xrk}
Tomoyuki Arakawa.
\newblock {Rationality of admissible affine vertex algebras in the category
  ${\mathcal{O}}$}.
\newblock {\em Duke Math. J.}, 165(1):67--93, 2016.

\bibitem[Ara18]{Arakawa:2018egx}
Tomoyuki Arakawa.
\newblock {Chiral algebras of class $\mathcal{S}$ and Moore-Tachikawa
  symplectic varieties}.
\newblock 11 2018.

\bibitem[AVEM21]{Arakawa:2021ogm}
Tomoyuki Arakawa, Jethro Van~Ekeren, and Anne Moreau.
\newblock {Singularities of nilpotent Slodowy slices and collapsing levels of
  W-algebras}.
\newblock 2 2021.

\bibitem[AY92]{Awata:1992sm}
Hidetoshi Awata and Yasuhiko Yamada.
\newblock {Fusion rules for the fractional level sl(2) algebra}.
\newblock {\em Mod. Phys. Lett. A}, 7:1185--1196, 1992.

\bibitem[Ber91]{Bershadsky:1990bg}
Michael Bershadsky.
\newblock {Conformal field theories via Hamiltonian reduction}.
\newblock {\em Commun. Math. Phys.}, 139:71--82, 1991.

\bibitem[BLL{\etalchar{+}}15]{Beem:2013sza}
Christopher Beem, Madalena Lemos, Pedro Liendo, Wolfger Peelaers, Leonardo
  Rastelli, and Balt~C. van Rees.
\newblock {Infinite Chiral Symmetry in Four Dimensions}.
\newblock {\em Commun. Math. Phys.}, 336(3):1359--1433, 2015.

\bibitem[BN16a]{Buican:2015tda}
Matthew Buican and Takahiro Nishinaka.
\newblock {Argyres-Douglas Theories, the Macdonald Index, and an RG
  Inequality}.
\newblock {\em JHEP}, 02:159, 2016.

\bibitem[BN16b]{Buican:2015ina}
Matthew Buican and Takahiro Nishinaka.
\newblock {On the superconformal index of Argyres\textendash{}Douglas
  theories}.
\newblock {\em J. Phys. A}, 49(1):015401, 2016.

\bibitem[BPRvR15]{Beem:2014rza}
Christopher Beem, Wolfger Peelaers, Leonardo Rastelli, and Balt~C. van Rees.
\newblock {Chiral algebras of class S}.
\newblock {\em JHEP}, 05:020, 2015.

\bibitem[BR18]{Beem:2017ooy}
Christopher Beem and Leonardo Rastelli.
\newblock {Vertex operator algebras, Higgs branches, and modular differential
  equations}.
\newblock {\em JHEP}, 08:114, 2018.

\bibitem[BS95]{bouwknegt1995w}
Peter Bouwknegt and Kareljan Schoutens.
\newblock {\em W-symmetry}, volume~22.
\newblock World Scientific, 1995.

\bibitem[CGS16]{Cordova:2016uwk}
Clay Cordova, Davide Gaiotto, and Shu-Heng Shao.
\newblock {Infrared Computations of Defect Schur Indices}.
\newblock {\em JHEP}, 11:106, 2016.

\bibitem[CGS17]{Cordova:2017mhb}
Clay Cordova, Davide Gaiotto, and Shu-Heng Shao.
\newblock {Surface Defects and Chiral Algebras}.
\newblock {\em JHEP}, 05:140, 2017.

\bibitem[CR12]{Creutzig:2012sd}
Thomas Creutzig and David Ridout.
\newblock {Modular Data and Verlinde Formulae for Fractional Level WZW Models
  I}.
\newblock {\em Nucl. Phys. B}, 865:83--114, 2012.

\bibitem[CR13a]{Creutzig:2013yca}
Thomas Creutzig and David Ridout.
\newblock {Modular Data and Verlinde Formulae for Fractional Level WZW Models
  II}.
\newblock {\em Nucl. Phys. B}, 875:423--458, 2013.

\bibitem[CR13b]{creutzig2013modular}
Thomas Creutzig and David Ridout.
\newblock {Modular data and Verlinde formulae for fractional level WZW models
  II}.
\newblock {\em Nuclear Physics B}, 875(2):423--458, 2013.

\bibitem[CS16]{Cordova:2015nma}
Clay Cordova and Shu-Heng Shao.
\newblock {Schur Indices, BPS Particles, and Argyres-Douglas Theories}.
\newblock {\em JHEP}, 01:040, 2016.

\bibitem[DLM96]{dong1996simple}
Chongying Dong, Haisheng Li, and Geoffrey Mason.
\newblock Simple currents and extensions of vertex operator algebras.
\newblock {\em Communications in mathematical physics}, 180(3):671--707, 1996.

\bibitem[DLM97]{dong1997vertex}
Chongying Dong, Haisheng Li, and Geoffrey Mason.
\newblock {Vertex operator algebras associated to admissible representations of
  $\widehat{sl(2)}$}.
\newblock {\em Communications in mathematical physics}, 184(1):65--93, 1997.

\bibitem[DLM00]{dong2000modular}
Chongying Dong, Haisheng Li, and Geoffrey Mason.
\newblock {Modular-Invariance of Trace Functions in Orbifold Theory and
  Generalized Moonshine}.
\newblock {\em Communications in Mathematical Physics}, 214(1):1--56, 2000.

\bibitem[ELS20]{Eager:2019zrc}
Richard Eager, Guglielmo Lockhart, and Eric Sharpe.
\newblock {Hidden exceptional symmetry in the pure spinor superstring}.
\newblock {\em Phys. Rev. D}, 101(2):026006, 2020.

\bibitem[Fas22]{Fasquel:2020iqf}
Justine Fasquel.
\newblock {Rationality of the Exceptional $\mathcal {W}$-Algebras $\mathcal
  {W}_k(\mathfrak {sp}_{4},f_{subreg})$ Associated with Subregular Nilpotent
  Elements of $\mathfrak {sp}_{4}$}.
\newblock {\em Commun. Math. Phys.}, 390(1):33--65, 2022.

\bibitem[FF90a]{Feigin:1990pn}
B.~Feigin and E.~Frenkel.
\newblock {Quantization of the Drinfeld-Sokolov reduction}.
\newblock {\em Phys. Lett. B}, 246:75--81, 1990.

\bibitem[FF90b]{feigin1990representations}
Boris~L Feigin and Edward~V Frenkel.
\newblock {Representations of affine Kac-Moody algebras, bosonization and
  resolutions}.
\newblock {\em letters in mathematical physics}, 19(4):307--317, 1990.

\bibitem[FHL93]{frenkel1993axiomatic}
Igor Frenkel, Yi-Zhi Huang, and James Lepowsky.
\newblock {\em {On axiomatic approaches to vertex operator algebras and
  modules}}, volume 494.
\newblock American Mathematical Soc., 1993.

\bibitem[FKR21]{Fehily:2020bif}
Zachary Fehily, Kazuya Kawasetsu, and David Ridout.
\newblock {Classifying Relaxed Highest-Weight Modules for Admissible-Level
  Bershadsky\textendash{}Polyakov Algebras}.
\newblock {\em Commun. Math. Phys.}, 385(2):859--904, 2021.

\bibitem[FS18]{Fluder:2017oxm}
Martin Fluder and Jaewon Song.
\newblock {Four-dimensional Lens Space Index from Two-dimensional Chiral
  Algebra}.
\newblock {\em JHEP}, 07:073, 2018.

\bibitem[Fuc89]{fuchs1989two}
Dmitrii~Borisovich Fuchs.
\newblock {Two projections of singular vectors of Verma modules over the affine
  Lie algebra $A_1^{(1)}$}.
\newblock {\em Funktsional'nyi Analiz i ego Prilozheniya}, 23(2):81--83, 1989.

\bibitem[FZ96]{FrenkelZhu}
I.~B. Frenkel and Y.~Zhu.
\newblock {Vertex operator algebras associated to representations of affine and
  Virasoro algebras}.
\newblock {\em Duke Mathematical Journal}, 66(1):123--168, 1996.

\bibitem[Gab01]{Gaberdiel:2001ny}
Matthias~R Gaberdiel.
\newblock {Fusion rules and logarithmic representations of a WZW model at
  fractional level}.
\newblock {\em Nucl. Phys. B}, 618:407--436, 2001.

\bibitem[Gai12]{Gaiotto:2009we}
Davide Gaiotto.
\newblock {$\mathcal{N}$=2 dualities}.
\newblock {\em JHEP}, 08:034, 2012.

\bibitem[GMN13]{Gaiotto:2009hg}
Davide Gaiotto, Gregory~W. Moore, and Andrew Neitzke.
\newblock {Wall-crossing, Hitchin systems, and the WKB approximation}.
\newblock {\em Adv. Math.}, 234:239--403, 2013.

\bibitem[KNS13]{Kaneko:2013uga}
Masanobu Kaneko, Kiyokazu Nagatomo, and Yuichi Sakai.
\newblock {Modular forms and second order ordinary differential equations:
  Applications to vertex operator algebras}.
\newblock {\em Lett. Math. Phys.}, 103:439--453, 2013.

\bibitem[KRW03]{kac2003quantum}
Victor~G Kac, Shi-Shyr Roan, and Minoru Wakimoto.
\newblock {Quantum reduction for affine superalgebras}.
\newblock {\em arXiv preprint math-ph/0302015}, 2003.

\bibitem[KRW22]{Kawasetsu:2021qls}
Kazuya Kawasetsu, David Ridout, and Simon Wood.
\newblock {Admissible-level $\mathfrak{sl}_3$ minimal models}.
\newblock {\em Letters in Mathematical Physics}, 112(5):96, 2022.

\bibitem[KW88]{kac1988modular}
Victor~G Kac and Minoru Wakimoto.
\newblock {Modular invariant representations of infinite-dimensional Lie
  algebras and superalgebras}.
\newblock {\em Proceedings of the National Academy of Sciences},
  85(14):4956--4960, 1988.

\bibitem[KW17]{kac2017remark}
Victor~G Kac and Minoru Wakimoto.
\newblock {A remark on boundary level admissible representations}.
\newblock {\em Comptes Rendus Mathematique}, 355(2):128--132, 2017.

\bibitem[Li96]{li1996local}
Hai-Sheng Li.
\newblock {local systems of twisted vertex operators}.
\newblock In {\em Moonshine, the Monster, and Related Topics: Joint Research
  Conference on Moonshine, the Monster, and Related Topics, June 18-23, 1994,
  Mount Holyoke College, South Hadley, Massachusetts}, volume 193, page 203.
  American Mathematical Soc., 1996.

\bibitem[Li23]{Hao2022}
Hao Li.
\newblock {Quasi-lisse vertex superalgebras}.
\newblock 2023.

\bibitem[LMRS02]{Lesage:2002ch}
F.~Lesage, P.~Mathieu, Jorgen Rasmussen, and H.~Saleur.
\newblock {The $\widehat{su(2)}_{-\frac{1}{2}}$ WZW model and the beta gamma
  system}.
\newblock {\em Nucl. Phys. B}, 647:363--403, 2002.

\bibitem[LP15]{Lemos:2014lua}
Madalena Lemos and Wolfger Peelaers.
\newblock {Chiral Algebras for Trinion Theories}.
\newblock {\em JHEP}, 02:113, 2015.

\bibitem[LXY22]{Li:2022njl}
Bohan Li, Dan Xie, and Wenbin Yan.
\newblock {On low rank 4d $\mathcal{N}=2$ SCFTs}.
\newblock {\em arXiv preprint arXiv:2212.03089}, 2022.

\bibitem[MFF86]{malikov1986singular}
Fedor~Georgievich Malikov, Boris~Lvovich Feigin, and Dmitry~B Fuks.
\newblock {Singular vectors in Verma modules over Kac—Moody algebras}.
\newblock {\em Functional Analysis and its Applications}, 20(2):103--113, 1986.

\bibitem[Miy00]{miyamoto2000modular}
Masahiko Miyamoto.
\newblock {A modular invariance on the theta functions defined on vertex
  operator algebras}.
\newblock {\em Duke Mathematical Journal}, 101(2):221--236, 2000.

\bibitem[Per08]{pervse2008vertex}
Ozren Per{\v{s}}e.
\newblock {Vertex operator algebras associated to certain admissible modules
  for affine Lie algebras of type A}.
\newblock {\em Glasnik matemati{\v{c}}ki}, 43(1):41--57, 2008.

\bibitem[Per13]{pervse2013note}
Ozren Per{\v{s}}e.
\newblock {A note on representations of some affine vertex algebras of type D}.
\newblock {\em Glasnik matemati{\v{c}}ki}, 48(1):81--90, 2013.

\bibitem[PP22]{Pan:2021mrw}
Yiwen Pan and Wolfger Peelaers.
\newblock {Exact Schur index in closed form}.
\newblock {\em Phys. Rev. D}, 106(4):045017, 2022.

\bibitem[Raz12]{Razamat:2012uv}
Shlomo~S. Razamat.
\newblock {On a modular property of N=2 superconformal theories in four
  dimensions}.
\newblock {\em JHEP}, 10:191, 2012.

\bibitem[Rid09]{Ridout:2008nh}
David Ridout.
\newblock {$\widehat{sl(2)}_{-\frac{1}{2}}$: A Case Study}.
\newblock {\em Nucl. Phys. B}, 814:485--521, 2009.

\bibitem[Rid10]{Ridout:2010qx}
David Ridout.
\newblock {$\widehat{sl(2)}_{-\frac{1}{2}}$ and the Triplet Model}.
\newblock {\em Nucl. Phys. B}, 835:314--342, 2010.

\bibitem[RY13]{Razamat:2013jxa}
Shlomo~S. Razamat and Masahito Yamazaki.
\newblock {S-duality and the N=2 Lens Space Index}.
\newblock {\em JHEP}, 10:048, 2013.

\bibitem[SS87]{Schwimmer:1986mf}
A.~Schwimmer and N.~Seiberg.
\newblock {Comments on the N=2, N=3, N=4 Superconformal Algebras in
  Two-Dimensions}.
\newblock {\em Phys. Lett. B}, 184:191--196, 1987.

\bibitem[SXY17]{Song:2017oew}
Jaewon Song, Dan Xie, and Wenbin Yan.
\newblock {Vertex operator algebras of Argyres-Douglas theories from
  M5-branes}.
\newblock {\em JHEP}, 12:123, 2017.

\bibitem[Wak01]{wakimoto2001infinite}
Minoru Wakimoto.
\newblock {\em {Infinite-dimensional Lie algebras}}, volume 195.
\newblock American Mathematical Soc., 2001.

\bibitem[Xie13]{Xie:2012hs}
Dan Xie.
\newblock {General Argyres-Douglas Theory}.
\newblock {\em JHEP}, 01:100, 2013.

\bibitem[Xu95]{xu1995intertwining}
XP~Xu.
\newblock {Intertwining operators for twisted modules of a colored vertex
  operator superalgebra}.
\newblock {\em Journal of Algebra}, 175(1):241--273, 1995.

\bibitem[XY21a]{Xie:2019vzr}
Dan Xie and Wenbin Yan.
\newblock {4d $\mathcal{N}=2$ SCFTs and lisse W-algebras}.
\newblock {\em JHEP}, 04:271, 2021.

\bibitem[XY21b]{Xie:2019zlb}
Dan Xie and Wenbin Yan.
\newblock {Schur sector of Argyres-Douglas theory and $W$-algebra}.
\newblock {\em SciPost Phys.}, 10(3):080, 2021.

\bibitem[XY21c]{Xie:2019yds}
Dan Xie and Wenbin Yan.
\newblock {W algebras, cosets and VOAs for 4d $ \mathcal{N} $ = 2 SCFTs from M5
  branes}.
\newblock {\em JHEP}, 04:076, 2021.

\bibitem[XYY21]{Xie:2016evu}
Dan Xie, Wenbin Yan, and Shing-Tung Yau.
\newblock {Chiral algebra of the Argyres-Douglas theory from M5 branes}.
\newblock {\em Phys. Rev. D}, 103(6):065003, 2021.

\bibitem[Yan17]{yang2017twisted}
Jinwei Yang.
\newblock Twisted representations of vertex operator algebras associated to
  affine lie algebras.
\newblock {\em Journal of Algebra}, 484:88--108, 2017.

\bibitem[Zhu22]{zhu2022bimodues}
Yiyi Zhu.
\newblock {Bimodues associated to twisted modules of vertex operator algebras}.
\newblock {\em arXiv preprint arXiv:2204.00238}, 2022.

\bibitem[ZPW22]{Zheng:2022zkm}
Haocong Zheng, Yiwen Pan, and Yufan Wang.
\newblock {Surface defects, flavored modular differential equations, and
  modularity}.
\newblock {\em Phys. Rev. D}, 106(10):105020, 2022.

\end{thebibliography}

\end{sloppypar}

\end{document}